\documentclass[12pt]{amsart}
\usepackage{amsmath}
\usepackage{geometry,amsthm,graphics,tabularx,amssymb,shapepar,eucal}
\usepackage{latexsym,amsfonts,amssymb,amsmath,amsxtra,bbm,color,mathrsfs}
\usepackage{amscd}
\usepackage{hyperref}
\usepackage{pdfsync}
\usepackage[all]{xy}
\usepackage{verbatim}
\usepackage{tikz-cd}

\usepackage{braket}

\newcommand{\bsl}{\backslash}

\newcommand{\BC}{{\mathbb {C}}}

\newcommand{\BN}{{\mathbb {N}}}

\newcommand{\BQ}{{\mathbb {Q}}}
\newcommand{\BR}{{\mathbb {R}}}

\newcommand{\CA}{{\mathcal {A}}}
\newcommand{\CB}{{\mathcal {B}}}
\newcommand{\CC}{{\mathcal {C}}}

\newcommand{\CF}{{\mathcal {F}}}
\newcommand{\CG}{{\mathcal {G}}}
\newcommand{\CH}{{\mathcal {H}}}

\newcommand{\CM}{{\mathcal {M}}}
\newcommand{\CN}{{\mathcal {N}}}

\newcommand{\CS}{{\mathcal {S}}}

\newcommand{\CU}{{\mathcal {U}}}
\newcommand{\CV}{{\mathcal {V}}}
\newcommand{\CW}{{\mathcal {W}}}
\newcommand{\CX}{{\mathcal {X}}}
\newcommand{\CY}{{\mathcal {Y}}}

\newcommand{\FS}{{\mathfrak {S}}}

\newcommand{\RO}{{\mathrm {O}}}

\newcommand{\RU}{{\mathrm {U}}}

\newcommand{\GL}{{\mathrm{GL}}}

\newcommand{\Ind}{{\mathrm{Ind}}}

\newcommand{\SL}{{\mathrm{SL}}}

\newcommand{\tr}{{\mathrm{tr}}}

\newcommand{\wt}{\widetilde}
\newcommand{\wh}{\widehat}

\newcommand{\bs}{\backslash}

\newcommand{\temp}{\operatorname{temp}}

\newcommand{\diag}{\operatorname{diag}}

\newcommand{\od}{\operatorname{d}}

\newcommand{\oL}{\operatorname{L}}

\newcommand{\oZ}{\operatorname{Z}}

\newcommand{\R}{\mathbb R}

\newcommand{\abs}[1]{\lvert#1\rvert}

\newcommand{\be}{\begin {equation}}
\newcommand{\ee}{\end {equation}}
\newcommand{\bee}{\begin {equation*}}
\newcommand{\eee}{\end {equation*}}

\newcommand{\cf}{\emph{cf.}~}

\theoremstyle{Theorem}

\theoremstyle{Theorem}

\theoremstyle{Theorem}
\newtheorem{lem}{Lemma}[section]
\newtheorem{corl}[lem]{Corollary}
\newtheorem{thml}[lem]{Theorem}
\newtheorem{leml}[lem]{Lemma}
\newtheorem{prpl}[lem]{Proposition}

\theoremstyle{Theorem}

\theoremstyle{Plain}

\newtheorem{remarkl}[lem]{Remark}

\theoremstyle{remark}

\theoremstyle{remark}

\theoremstyle{Definition}

\newcommand{\Irr}{\mathrm{Irr}}

\newcommand{\bk}{\mathbbm{k}}

\numberwithin{equation}{section}

\begin{document}

\title[Jacquet-Shalika integrals]
{Local Jacquet-Shalika integrals and modifying factors} 

\author[D. Jiang]{Dihua Jiang}
\address{School of Mathematics, University of Minnesota, Minneapolis, MN 55455, USA}
\email{dhjiang@math.umn.edu}

\author[D. Liu]{Dongwen Liu}
\address{School of Mathematical Sciences, Zhejiang University, Hangzhou, 310058, P. R. China}
\email{maliu@zju.edu.cn}

\author[B. Sun]{Binyong Sun}
\address{Institute for Advanced Study in Mathematics and New Cornerstone Science Laboratory, Zhejiang University, Hangzhou, 310058, P. R. China}
\email{sunbinyong@zju.edu.cn}

\author[F. Tian]{Fangyang Tian}
\address{School of Mathematical Sciences, Zhejiang University, Hangzhou, 310058, P. R. China}
\email{tianfangyangmath@zju.edu.cn}



\begin{abstract}
Over any local field of characteristic zero,
we develop a local theory of Jacquet–Shalika integrals that were originally introduced by Jacquet and Shalika to study  exterior square local $L$-factors for \(\mathrm{GL}_m\). In particular, at least in the archimedean case, we establish  a complete functional equation matching Artin local \(\varepsilon\)-factors. Central to the argument is a novel comparison strategy between  Jacquet–Shalika  integrals and open-orbit integrals associated to the  Shalika subgroup.
\end{abstract}

\maketitle

\tableofcontents

\section{Introduction}

In \cite{JS90}, Jacquet and Shalika introduce certain local integrals to study the exterior square $L$-functions. These integrals are now called the Jacquet--Shalika integrals. In this paper, we will establish some basic properties of these integrals. These properties will be used for the arithmetic study of some $L$-functions (see \cite{JLST26}). 

\subsection{Representations and exterior square local factors}

Let $\bk$ be a local field of characteristic zero, with normalized absolute value $|\cdot|_{\bk}$. Fix a nontrivial unitary character $\psi :\bk\rightarrow \mathbb{C}^\times$.

For a connected reductive group $G$ over $\bk$, denote by $\operatorname{Irr}(G)$ the set of isomorphism classes of irreducible admissible smooth representations of $G(\bk)$ (in the archimedean case, ``admissible smooth representations'' means Casselman--Wallach representations). 
Here and henceforth, when no confusion is possible, we do not distinguish a reductive group over $\bk$ from the group of its $\bk$-points. 
We also do not distinguish an element of $\operatorname{Irr}(G)$ from an irreducible representation representing it. 
Let $\operatorname{Irr}_{\mathrm{temp}}(G)\subset \operatorname{Irr}(G)$ denote the subset of the tempered unitarizable representations.

Let $m\in \BN:=\{0,1,2,\dots\}$. We introduce the following groups, which will be used throughout the paper:
\begin{itemize}
    \item $G_m:=\GL_m(\bk)$;
    \item $B_m=A_m N_m\subset G_m$ is the the upper triangular Borel subgroup, where $A_m$ is the diagonal torus and  $N_m$ is the unipotent radical;
    \item $\overline{B}_m\subset G_m$ denotes the lower triangular Borel subgroup;
    \item  $P_m$ denotes the mirabolic subgroup of $G_m$, i.e., the subgroup of matrices with last row 
\be\label{em}
e_m:=(0,0,\dots, 0, 1) \in \bk^m;
\ee
  \item 
$K_m$ denotes the standard maximal compact subgroup of $G_m$, namely 
\[
K_m =
\begin{cases}
\RO(m), & \text{if } \bk \cong \R,\\[2pt]
\RU(m), & \text{if } \bk \cong \BC,\\[2pt]
\GL_m(\mathcal{O}_\bk), & \text{if } \bk \text{ is non-archimedean with ring of integers } \mathcal{O}_\bk.
\end{cases}
\]
\end{itemize}
 Define a character 
\[
\psi_m:  N_m\to \mathbb{C}^\times,\quad [x_{i,j}]_{1\leq i,j\leq  m}\mapsto \psi\left(\sum^{m-1}_{i=1}x_{i,i+1}\right).
\]
To shorten the notation, in this paper we write 
\[
\omega(g) := \omega(\det g)\quad \text{for every character } \omega:\bk^\times\rightarrow \mathbb{C}^\times \text{ and every } g\in G_m.
\]
In particular, $|g|_{\bk} := |\det g|_{\bk}$. 

We consider a representation of $G_m$ given by the normalized smooth parabolic induction 
\begin{equation} \label{nt}
\pi_\lambda =  \operatorname{Ind}^{G_m}_{\overline{P}}(\tau_\lambda) = \operatorname{Ind}^{G_m}_{\overline{P}} \bigl( \tau_1 |\cdot|^{\lambda_1}_{\bk} \,\widehat{\otimes} \, \tau_2|\cdot|_{\bk}^{\lambda_2}\,\widehat{\otimes} \cdots \widehat{\otimes}\, \tau_r|\cdot|_{\bk}^{\lambda_r}\bigr),
\end{equation}
where 
\begin{itemize}
\item $\overline{P}$ is the block lower triangular parabolic subgroup of $G_m$ with Levi subgroup 
\[
M =   G_{n_1} \times G_{n_2}\times \cdots\times G_{n_r},
\]
where $r\geq 0$, $n_1,n_2,\ldots, n_r\geq 1$, and $n_1+n_2+\cdots+n_r=m$;
\item
 $\tau = \tau_1\,\widehat{\otimes} \, \tau_2\, \widehat{\otimes} \cdots \widehat{\otimes} \, \tau_r \in \operatorname{Irr}_{\mathrm{temp}}(M)$,  where  $\tau_i\in\operatorname{Irr}_{\mathrm{temp}}(G_{n_i})$ ($i=1,2,\dots, r$);
 \item  $\lambda = (\lambda_1, \lambda_2, \dots, \lambda_r)  \in \mathbb{C}^r$; and  
 \item  $\tau_{\lambda}=\tau_1 |\cdot|^{\lambda_1}_{\bk} \,\widehat{\otimes} \, \tau_2|\cdot|_{\bk}^{\lambda_2}\,\widehat{\otimes} \cdots \widehat{\otimes}\, \tau_r|\cdot|_{\bk}^{\lambda_r}\in\operatorname{Irr}(M)$.
 \end{itemize}
Here $\widehat{\otimes}$ stands for the completed inductive tensor product. In the non-archimedean case, the smooth representations are equipped with the finest locally convex topologies and the completed inductive tensor product agrees with the algebraic tensor product. 

The following facts are well-known.
 \begin{itemize}
  \item  The representation $\pi_\lambda$ is of Whittaker type in the sense that 
  \[
  \dim \operatorname{Hom}_{N_m}(\pi_\lambda, \psi_m)=1.
  \]
  \item For fixed $\tau \in \operatorname{Irr}_{\mathrm{temp}}(M)$, $\pi_\lambda$ is irreducible for $\lambda$ outside a measure zero subset of  $\mathbb{C}^r$.
  \item Each $\pi \in \operatorname{Irr}_{\mathrm{gen}}(G_m)$ (the subset of generic representations in $\operatorname{Irr}(G_m)$) is isomorphic to an induced representation $\pi_\lambda$ of the form \eqref{nt}.
 \end{itemize}

Denote by $\Re(z)$ the real part of a complex number $z$. We will use the following notation when $m>0$ (or equivalently $r>0$): for $\lambda=(\lambda_1,\lambda_2, \dots, \lambda_r) \in \mathbb{C}^r$, write
\begin{equation} \label{minmax}
\min\Re(\lambda):=\min_{i=1,2,\dots, r} \Re(\lambda_i),\quad \max\Re(\lambda) :=\max_{i=1,2,\dots, r} \Re(\lambda_i).
\end{equation}
Following \cite{BP21}, an admissible smooth representation of $G_m$ is called \emph{nearly tempered} if it is isomorphic to some $\pi_\lambda$ as in \eqref{nt} with  $|\Re(\lambda_i)|<1/4$ for all $i=1, 2, \ldots, r$.
It is known that all nearly tempered representations are irreducible.

For $\pi\in \operatorname{Irr}(G_m)$, denote by $\phi_\pi$ the local $L$-parameter of $\pi$ under the local Langlands correspondence, which is an $m$-dimensional admissible representation of the Weil–Deligne group $W_{\bk}'$ of $\bk$.  
Fix a character $\eta:\bk^\times\rightarrow \mathbb{C}^\times$.  We have the twisted exterior‐square local factors (see \cite{CST17, Sh24})
\begin{equation} \label{exL}
\begin{aligned}
& \oL(s, \pi, \wedge^2\otimes \eta^{-1}) := \oL(s, \wedge^2\phi_\pi\otimes \eta^{-1}),\\
& \varepsilon(s, \pi, \wedge^2\otimes\eta^{-1}, \psi) := \varepsilon(s, \wedge^2\phi_\pi\otimes\eta^{-1}, \psi),\\
& \gamma(s, \pi, \wedge^2\otimes \eta^{-1},\psi) :=  \varepsilon(s, \pi, \wedge^2\otimes\eta^{-1}, \psi)\cdot  \frac{\oL(1-s, \pi^\vee, \wedge^2\otimes \eta)}{\oL(s, \pi, \wedge^2\otimes \eta^{-1})},
\end{aligned}
\end{equation}
where the right‐hand sides are as in \cite{T79}.
For the parabolic induction $\pi_\lambda$ in \eqref{nt}, we define 
\begin{equation} \label{exL2nt}
\begin{aligned}
\oL(s, \pi_\lambda, \wedge^2\otimes\eta^{-1})
  := & \prod_{i=1}^r \oL(s+2\lambda_i, \wedge^2\phi_{\tau_i}\otimes\eta^{-1}) \\
  & \qquad \cdot \prod_{1\leq j<k \leq r} \oL(s+\lambda_j+\lambda_k, \phi_{\tau_j}\otimes\phi_{\tau_k}\otimes\eta^{-1}),
\end{aligned}
\end{equation}
and $\varepsilon(s, \pi_\lambda, \wedge^2\otimes\eta^{-1},\psi)$ and $\gamma(s, \pi_\lambda, \wedge^2\otimes \eta^{-1})$ are defined analogously. 

By the compatibility of the local Langlands correspondence with parabolic induction and unramified twists, if $\pi_\lambda^0$ denotes the Langlands subquotient of $\pi_\lambda$, then
\[
 \oL(s, \pi_\lambda,  \wedge^2\otimes \eta^{-1}) = \oL(s, \pi_\lambda^0,  \wedge^2\otimes \eta^{-1}),\quad \varepsilon(s, \pi_\lambda, \wedge^2\otimes\eta^{-1},\psi) = 
\varepsilon(s, \pi_\lambda^0, \wedge^2\otimes\eta^{-1},\psi),
\] 
where the right‐hand sides are given by \eqref{exL}. In particular, \eqref{exL} and \eqref{exL2nt} coincide when $\pi_\lambda$ is irreducible.

\subsection{Jacquet-Shalika integrals} \label{sec1.2.2}

Fix the self-dual Haar measure on $\bk$ with respect to $\psi$. For $k, k'\in 
\BN$, denote by $\bk^{k \times k'}$ the space of 
$k\times k'$ matrices over $\bk$, and write $M_k := \bk^{k\times k}$. We endow $\bk^{k\times k'}$ with the product measure, and fix the Haar measure on $G_k$ to be 
\[
\od\! g = |g|_{\bk}^{-n}\cdot \prod_{i,j=1}^{k} \od\! g_{i,j}
\quad\text{for } g=[g_{i,j}]_{1\leq i,j\leq k}\in G_k.
\]
For $\phi \in \mathcal{S}(\bk^k)$, the space of Schwartz functions on $\bk^k := \bk^{1\times k}$, define its Fourier transform with respect to a nontrivial unitary character $\psi'$ of $\bk$ by
\begin{equation}\label{FT}
  \mathcal{F}_{\psi'}(\phi)(x) = \int_{\bk^k} \phi(y) \,\psi'(y\, {}^t\!x)\od\! y, \qquad x\in \bk^k.  
\end{equation}
Here and thereafter, ${}^t\!(\cdot)$ indicates the transpose of a matrix.  

Throughout  the rest of this paper, set  \[
n:=\lfloor m/2\rfloor\quad\textrm{ so that $m=2n$ or $2n+1$.}
\]
The Shalika subgroup $S_m$ of $G_m$ is defined by
\[
S_m:=
\begin{cases}
\left\{
\begin{bmatrix} g & Xg \\ 0 & g \end{bmatrix}
\;\middle|\; g\in G_n,\ X\in M_n
\right\}, & m=2n, \\[6pt]
\left\{
\begin{bmatrix} g & Xg & y \\ 0 & g & 0 \\ 0 & xg & 1 \end{bmatrix}
\;\middle|\;
\begin{array}{l}
g\in G_n,\ X\in M_n, \\
y\in \bk^{n\times 1},\ x\in \bk^{1\times n}
\end{array}
\right\}, & m=2n+1,
\end{cases}
\]
which is a unimodular group.  
The Haar measures on $S_m$, $N_m$, etc., are induced from the fixed Haar measures on $G_n$ (viewed as a subgroup of $S_m$) and $\bk$. We always take right invariant quotient measures (when such measures exist) on homogeneous spaces under right actions of locally compact groups.

In the following, we introduce a representation $R_{\varphi_m}$ of $S_m$, where $\varphi_m$ is a certain character determined by $\eta$ and $\psi$. Similarly, one can define a representation $R_{\varphi_m^{-1}}$, the details of which will be omitted.

 If $m=2n$ is even, we first define a character 
\begin{equation} \label{varphi2n}
 \varphi_{2n}: S_{2n}\to \mathbb{C}^\times, \quad \begin{bmatrix} g& Xg  \\ 0& g\end{bmatrix} \mapsto \eta(g) \psi(\operatorname{tr} X).
\end{equation}
Let $S_{2n}$ act on $\bk^n$ from the right by
\begin{equation} \label{Sact}
(v,h)\mapsto v\cdot h:=vg, \quad \text{where } v\in \bk^n,\, h = \begin{bmatrix} g & Xg   \\ 0& g \end{bmatrix}\in S_{2n}.
\end{equation}
Then we define a representation $R_{\varphi_{2n}}$ of $S_{2n}$ on $\mathcal{S}(\bk^n)$ by 
\begin{equation} \label{Seven}
\left(R_{\varphi_{2n}}(h)\phi\right)(v) :=  \varphi_{2n}(h) \phi(v\cdot h),
\quad h\in S_{2n},\, \phi \in \mathcal{S}(\bk^n), \, v\in \bk^n.
\end{equation}

If $m=2n+1$ is odd, we first define a character 
\begin{equation}\label{varphi2n1}
    \varphi_{2n+1}: S_{2n+1}\cap P_{2n+1}\to \mathbb{C}^\times,\quad  \begin{bmatrix} g& Xg  & y \\ 0& g &  0\\ 0&0 & 1\end{bmatrix} \mapsto \eta(g) \psi(\operatorname{tr} X),
\end{equation}
where $P_m$ is the mirabolic subgroup of $G_m$ as introduced before. 
Then we define 
\begin{equation}\label{Sodd}
  R_{\varphi_{2n+1}} := \operatorname{ind}^{S_{2n+1}}_{S_{2n+1}\cap P_{2n+1}}\varphi_{2n+1}
\quad\text{(the unnormalized Schwartz induction)},  
\end{equation}
which is also realized on the space $\mathcal{S}(\bk^n)$ (see Section \ref{sec2.2} for details).

For each $k\in \BN$, let $\mathfrak{S}_k$ denote the permutation group of the set $\{1,2,\dots,k\}$ (this is the empty set when $k=0$). We identify a permutation $\sigma\in \mathfrak{S}_k$ with the permutation matrix in $G_k$ whose $(\sigma(i),i)$-th entry is $1$ for all $i=1,2,\dots,k$. Then $\mathfrak{S}_k$ is identified with the group of permutation matrices in $G_k$. Let $w_k$ be the longest element of $\mathfrak{S}_k$, which is identified as  the $k\times k$ anti-diagonal permutation matrix. 
Introduce as in \cite{JS90} the following element of $\mathfrak{S}_m$:
\begin{equation} \label{sigmam}
  \sigma_m := \begin{cases} 
  \begin{pmatrix} 
  1 & 2 & \cdots & n & n+1 & n+2 & \cdots & 2n \\
  1 & 3 & \cdots & 2n-1 & 2 & 4 & \cdots & 2n 
  \end{pmatrix}, &  m=2n, \\[6pt]
  \begin{pmatrix} 
  1 & 2 & \cdots & n & n+1 & n+2 & \cdots & 2n & 2n+1 \\
  1 & 3 & \cdots & 2n-1 & 2 & 4 & \cdots & 2n & 2n+1  
  \end{pmatrix}, &  m=2n+1.
  \end{cases} 
\end{equation}

Let $\pi_\lambda = \operatorname{Ind}^{G_m}_{\overline P}(\tau_\lambda)$ ($\lambda\in\mathbb{C}^r$) be an induced representation of $G_m$ as in \eqref{nt}. 
Fix a nonzero Whittaker functional  
\[
\nu_{\tau,\psi}\in \operatorname{Hom}_{N_m\cap M}(\tau, \psi_m|_{N_m\cap M}),
\]
which is unique up to scalar multiplication. 
For each $\lambda\in \mathbb{C}^r$, it determines a non-zero Whittaker functional 
\[
J_{\lambda}\in \operatorname{Hom}_{N_m}(\pi_\lambda, \psi_m) 
\]
such that 
\begin{equation}\label{whittaker1}
J_{\lambda}(\varphi)=\int_{(N_m\cap M)\backslash N_m} \langle \nu_{\tau,\psi}, \varphi(x) \rangle \cdot \overline{\psi_m}(x) \,\od\! x
\end{equation}
for all $\varphi\in \pi_\lambda$ whose support is contained in the Zariski open subset $\overline P N_m$ of $G_m$. 

Denote by 
\begin{equation}\label{whittaker2}
\mathcal{W}(\pi_\lambda,\psi_m):=\{ W_\varphi \mid \varphi\in \pi_\lambda \}
\subset \operatorname{Ind}_{N_m}^{G_m} \psi_m \qquad(\text{smooth induction})
\end{equation}
the Whittaker model of $\pi_\lambda$ with respect to $(N_m, \psi_m)$, where
\[
W_{\varphi}(g)=\langle J_{\lambda}, g\cdot \varphi \rangle, \quad \text{for all } g\in G_m.
\]
We equip $\mathcal{W}(\pi_{\lambda},\psi_m)$ with the quotient topology of $\pi_{\lambda}$. In general, for every admissible smooth representation $\pi$ of $G_m$ such that $\dim \operatorname{Hom}_{N_m}(\pi, \psi_m)=1$, we can similarly define its Whittaker model $\mathcal{W}(\pi, \psi_m)\subset \operatorname{Ind}_{N_m}^{G_m} \psi_m$.

For $W \in \operatorname{Ind}_{N_m}^{G_m} \psi_m$, $\phi \in \mathcal{S}(\bk^n)$, and $s\in\mathbb{C}$, the \emph{Jacquet--Shalika integral} introduced in \cite{JS90} can be reformulated as 
\begin{equation} \label{JSint}
 \oZ_{\mathrm{JS}}(s, W, \phi, \varphi_m^{-1}) := 
 \begin{cases}
 \displaystyle \int_{\overline{S}_m} W(\sigma_m h) \cdot \bigl(R_{\varphi_m^{-1}}(h)\phi\bigr)(e_n) \cdot | h|_{\bk}^{\frac{s}{2}}\,\od\! h, & m=2n, \\[8pt]
 \displaystyle \int_{\overline{S}_m} W(\sigma_m h) \cdot \bigl(R_{\varphi_m^{-1}}(h)\phi\bigr)(0) \cdot | h|_{\bk}^{\frac{s}{2}}\,\od\! h, & m=2n+1,
 \end{cases}
\end{equation}
where $e_n=(0,0,\dots,0,1)\in \bk^n$ as in \eqref{em}, and 
\begin{equation}\label{defsm}
\overline{S}_m:= \left(\sigma_m^{-1}N_m \sigma_m \cap S_m \right)\setminus S_m. 
\end{equation}
If moreover $W\in \mathcal{W}(\pi_\lambda, \psi_m)$, then we write 
\begin{equation}\label{zcirc}
\oZ_{\mathrm{JS}}^\circ(s, W, \phi, \varphi_m^{-1}) :=  \frac{\oZ_{\mathrm{JS}}(s, W, \phi, \varphi_m^{-1})}{\oL(s, \pi_\lambda, \wedge^2\otimes\eta^{-1})}.
\end{equation}

\subsection{Basic properties of Jacquet--Shalika integrals}

Write
\[
  \pi_\tau:=\operatorname{Ind}^{K_m}_{K_m\cap \overline P}(\tau),
\]
which is a smooth representation of $K_m$.  For every $\lambda\in \mathbb{C}^r$, restriction of functions induces an identification 
\be\label{indke}
  (\pi_\lambda)|_{K_m}=\pi_\tau. 
\ee
For each $f\in \pi_\tau$, write $f_\lambda$ for the corresponding element in $\pi_\lambda$.

 For every $W\in \mathcal{W}(\pi_\lambda, \psi_m)$, write 
\[
\widetilde{W}(h) := W(w_m \, {}^t h^{-1}),\quad h\in G_m.
\]
Then $\widetilde W\in \mathcal{W}(\breve{\pi}_\lambda, \overline{\psi}_m)$, where $\breve{\pi}_\lambda$ denotes the MVW involution 
of $\pi_\lambda$ (see \cite{MVW87}) so that
\[
\breve{\pi}_\lambda(h):=\pi_\lambda(w_m \, {}^t h^{-1} w_m) \quad \text{for all } h\in G_m.
\]
In particular, $\breve{\pi}_\lambda \cong \pi_\lambda^\vee$ when $\pi_\lambda$ is irreducible. Here and henceforth, a superscript $^\vee$ indicates the contragredient representation. 

Introduce the following element of $\mathfrak{S}_m$:
\begin{equation} \label{taum}
 \sigma'_m := 
 \begin{cases}
   \begin{bmatrix} 0 & 1_n \\ 1_n & 0 \end{bmatrix}, & m=2n, \\[6pt]
   \begin{bmatrix} 0 & 1_n & 0 \\ 1_n & 0 & 0 \\ 0 & 0 & 1 \end{bmatrix}, & m=2n+1,
 \end{cases}
\end{equation}
where $1_n$ denotes the $n\times n$ identity matrix. 
Denote by $\widehat{\bk^\times}$ the set of characters of $\bk^\times$, and for any $\omega\in \widehat{\bk^\times}$ let $\Re(\omega)$ be the real number (denoted by $\operatorname{ex}(\omega)$ in \cite{LLSS23}) such that 
\[
|\omega(a)| = |a|_{\bk}^{\Re(\omega)} \quad \text{for all } a\in \bk^\times. 
\]

Our main result on the local theory of Jacquet--Shalika integrals is as follows.

\begin{thml}\label{thm:FE_m}  
Let $\bk$ be any local field of characteristic zero and  
the representation $\pi_\lambda = \operatorname{Ind}^{G_m}_{\overline{P}}(\tau_\lambda)$ is defined as in \eqref{nt} with $\tau\in \Irr_{\temp}(M)$. Assume that the cuspidal support of $\tau$ is represented by a character of $(\bk^\times)^m$ when $\bk$ is non-archimedean.
\begin{enumerate}
    \item Suppose that $(s,\lambda)\in \mathbb{C}\times \mathbb{C}^r$, and  \begin{equation}\label{ss'}
    \Re(s)>\Re(\eta)-2\min\Re(\lambda)\quad\textrm{when $m>1$}.
\end{equation}
    Then for every $W\in \mathcal{W}(\pi_\lambda, \psi_m)$ and $\phi \in \mathcal{S}(\bk^n)$, the integral $\oZ_{\mathrm{JS}}(s, W, \phi, \varphi_m^{-1})$ converges absolutely.

    \item There exists a unique continuous function
    \[
    \oZ^\circ(\,\cdot\, ,\,\cdot, \,\cdot\, ,\, \cdot\, ): \mathbb{C}\times \mathbb{C}^r\times \pi_\tau\times \mathcal{S}(\bk^n)\rightarrow \mathbb{C}
    \]
    that is holomorphic in the first two variables and linear in the last two variables such that
    \[
    \oZ^\circ(s,\lambda, f, \phi) = \oZ^\circ_{\mathrm{JS}}(s, W_{f_\lambda}, \phi, \varphi_m^{-1})
    \]
    for all $f\in \pi_\tau$, $\phi \in \mathcal{S}(\bk^n)$, and $(s,\lambda)\in \mathbb{C}\times \mathbb{C}^r$ satisfying \eqref{ss'}.

    \item Let $(s,\lambda)\in \mathbb{C}\times \mathbb{C}^r$. For every $W\in \mathcal{W}(\pi_\lambda, \psi_m)$ and $\phi \in \mathcal{S}(\bk^n)$, the following functional equation holds:
    \begin{equation}\label{eq:FE_m}
    \oZ^\circ_{\mathrm{JS}}(1-s, \sigma'_m\cdot  \widetilde W, \widehat{\phi}, \varphi_m) = \eta(-1)^{mn} \varepsilon(s, \pi_\lambda, \wedge^2\otimes \eta^{-1},\psi) \oZ^\circ_{\mathrm{JS}}(s, W, \phi, \varphi_m^{-1}),
    \end{equation}
    where
    \[
    \widehat{\phi} :=
    \begin{cases}
    \mathcal{F}_{\psi}(\phi), & \text{if } m \text{ is even},\\
    \mathcal{F}_{\overline{\psi}}(\phi), & \text{if } m \text{ is odd}.
    \end{cases}
    \]

    \item Suppose that  $\lambda\in \mathbb{C}^r$, and  
    \[
\max\Re(\lambda)< \min\Re(\lambda)+1/2\quad \textrm{when }m>1.
\]
Then for every $s_0\in\mathbb{C}$ there exist $W\in \mathcal{W}(\pi_\lambda,\psi_m)$ and $\phi\in \mathcal{S}(\bk^n)$ such that
    \[
    \oZ^\circ_{\mathrm{JS}}(s_0, W, \phi, \varphi_m^{-1}) \neq 0.
    \]
\end{enumerate}
\end{thml}

As in \eqref{zcirc}, the left-hand side of \eqref{eq:FE_m} is defined by
\[
\oZ^\circ_{\mathrm{JS}}(1-s, \sigma'_m\cdot \widetilde W, \widehat{\phi}, \varphi_m) := \frac{\oZ_{\mathrm{JS}}(1-s, \sigma'_m\cdot \widetilde W, \widehat{\phi}, \varphi_m)}{\oL(1-s, \pi_\lambda^\vee,\wedge^2\otimes\eta)}\in \mathbb{C},
\]
where $\oZ_{\mathrm{JS}}(1-s, \sigma'_m\cdot \widetilde W, \widehat{\phi}, \varphi_m)$ is defined analogously to \eqref{JSint} and converges absolutely when 
\[
\Re(1-s) > 2\max\Re(\lambda) - \Re(\eta)\quad (\textrm{for $m>1$}).
\]
Note that $\breve{\pi}_\lambda$ and $\pi_\lambda^\vee$ are also induced representations of the form \eqref{nt}, and that \[
\oL(s, \breve{\pi}_\lambda, \wedge^2\otimes\eta) = \oL(s, \pi_\lambda^\vee, \wedge^2\otimes\eta) \qquad \textrm{(defined as in \eqref{exL2nt}).}
\]

Recall that every $\pi \in \operatorname{Irr}_{\mathrm{gen}}(G_m)$ is isomorphic to an induced representation  of the form \eqref{nt}, for which we can similarly define
\[
\oZ^\circ_{\mathrm{JS}}(s, W, \phi, \varphi_m^{-1}) = \frac{\oZ_{\mathrm{JS}}(s, W, \phi, \varphi_m^{-1}) }{\oL(s, \pi, \wedge^2\otimes\eta^{-1})},\quad W\in \mathcal{W}(\pi,\psi_m), \ \phi\in \mathcal{S}(\bk^n).
\]
In particular, we have the following:
\begin{itemize}
\item Theorem \ref{thm:FE_m} applies to all $\pi \in \operatorname{Irr}_{\mathrm{gen}}(G_m)$ when $\bk$ is archimedean.
\item If $\pi_\lambda\otimes |\eta|^{-\frac{1}{2}}$ is nearly tempered, then the condition in Theorem \ref{thm:FE_m} (4) clearly holds. Here $|\eta|^{-\frac{1}{2}}$ denotes the character $|\eta(\det(\cdot))|^{-\frac{1}{2}}$ of $G_m$.
\end{itemize}

\begin{remarkl}

For non-archimedean local fields and trivial $\eta$, the absolute convergence (for sufficiently large $\Re(s)$), meromorphic continuation, and functional equation of the Jacquet--Shalika integrals \eqref{JSint} were established in \cite{KR12, M14, CM15, Jo20}. However, it remained an open question whether the local exterior square $\varepsilon$-factors arising from that functional equation coincide with the Artin local factors in \eqref{exL} (see \cite{CST17, Sh24}).

In the archimedean setting, much less was known about these integrals. The archimedean theory developed in Theorem~\ref{thm:FE_m} is complete except for the non-vanishing of the normalized integrals; nonetheless, the non-vanishing result in part~(4) covers the nearly tempered case, which is sufficient for the applications in \cite{JLST26}. Importantly, our method recovers the expected Artin local factors, a feature essential for arithmetic applications.



\end{remarkl}

\subsection{Open orbit integrals and modifying factors}

Our proof of Theorem~\ref{thm:FE_m} is purely local and follows the strategy of \cite{LLSS23}, which studies modifying factors for the Rankin--Selberg convolution case. The idea is to compare the Jacquet--Shalika integrals for principal series representations with integrals over the open orbit under the action of the Shalika subgroup $S_m$ on a certain variety, as given in Theorem~\ref{thm:MF_m}. 

Such a comparison yields certain modifying factors (see \eqref{mfactor}), which in the non-archimedean case are compatible with the conjecture for $p$-adic $L$-functions due to Coates and Perrin-Riou \cite{CPR89, C89}. This type of phenomenon has been observed for several families of period integrals (see \cite{LSS21, LLSS23, LS25}).

To explain the details, we introduce an $S_m$-variety $\mathcal{X}_m$ as follows. Let $\mathcal{B}_m := \overline{B}_m \backslash G_m$ be the flag variety on which $G_m$ acts from the right. Define $\mathcal{X}_m := \mathcal{B}_m \times \bk^n$.

We have specified a right action of $S_m$ on $\bk^n$ when $m$ is even in \eqref{Sact}. If $m = 2n+1$, then the right action of $S_m$ on $\bk^n$ is given by
\begin{equation} \label{Sact'}
(v,h) \mapsto v \cdot h := (v+x)g,
\quad \text{where } v \in \bk^n,\;
h = \begin{bmatrix}
g & Xg & y \\
0 & g & 0 \\
0 & xg & 1
\end{bmatrix} \in S_{2n+1}.
\end{equation}
The diagonal action of $S_m$ on $\mathcal{X}_m$ has a unique Zariski-open orbit, with base point
\begin{equation} \label{basept}
x_m :=
\begin{cases}
(\overline{B}_m z_m, v_n), & m = 2n, \\
(\overline{B}_m z_m, 0), & m = 2n+1,
\end{cases}
\end{equation}
where
\be\label{vn}
v_n := (1,1,\dots,1) \in \bk^n,
\ee
and
\begin{equation} \label{zm}
z_m :=
\begin{cases}
\begin{bmatrix} 1_n & 0 \\ 0 & w_n \end{bmatrix}, & m = 2n, \\[1em]
\begin{bmatrix} 1_n & 0 & 0 \\ 0 & w_n & {}^t v_n \\ 0 & 0 & 1 \end{bmatrix}, & m = 2n+1.
\end{cases}
\end{equation}
Moreover, the stabilizer of $x_m$ in $S_m$ is trivial.

Let $\xi = (\xi_1,\xi_2,\dots,\xi_m) \in (\widehat{\bk^\times})^m$. View it as a character of $\overline{B}_m$ in the obvious way, and put
\be\label{defxi}
I(\xi) := \operatorname{Ind}^{G_m}_{\overline{B}_m}(\xi)
\qquad(\text{normalized smooth induction}).
\ee
For $f \in I(\xi)$, $\phi \in \mathcal{S}(\bk^n)$, and $s \in \mathbb{C}$, formally define the \emph{open orbit integral}
\begin{equation} \label{openint}
\Lambda_{\mathrm{JS}}(s, f, \phi, \varphi_m^{-1}) :=
\begin{cases}
\displaystyle \int_{S_m} f(z_m h) \cdot \bigl(R_{\varphi_m^{-1}}(h) \phi\bigr)(v_n) \cdot |h|_{\bk}^{s/2} \,\od\! h, & m = 2n, \\[8pt]
\displaystyle \int_{S_m} f(z_m h) \cdot \bigl(R_{\varphi_m^{-1}}(h) \phi\bigr)(0) \cdot |h|_{\bk}^{s/2} \,\od\!h, & m = 2n+1.
\end{cases}
\end{equation}

Denote by $W_f \in \mathcal{W}(I(\xi), \psi_m)$ the Whittaker function associated to $f$ and $\psi_m$ via the Jacquet integral
\be\label{wf}
W_f(g) = \int_{N_m} f(ug) \overline{\psi_m}(u) \,\od\! u, \quad g\in G_m
\ee
in the sense of holomorphic continuation (see \cite{J67} and \cite[Theorem 15.4.1]{W92} for details).
Put 
\[
\widetilde{\xi} := (\xi_m^{-1}, \dots, \xi_1^{-1})\in (\widehat{\bk^\times})^m.
\]
For every $f \in I(\xi)$, write
\begin{equation}\label{wtf}
  \widetilde{f}(h) := f(w_m \, {}^t h^{-1}), \quad h \in G_m\ {\rm and}\ w_m\ {\rm is\ the\ longest\ element\ in}\ \FS_m.  
\end{equation}
Then $\widetilde{f} \in I(\widetilde{\xi})$ and
\[
\widetilde{W_f} = W_{\widetilde{f}} \in \mathcal{W}(I(\widetilde{\xi}), \overline{\psi}_m).
\]
Here and thereafter, by abuse of notation, we write $W_{\widetilde{f}}$ for the Whittaker function associated to $\widetilde{f}$ and $\overline{\psi}_m$; this should not cause any confusion.

The group $(\widehat{\bk^\times})^m$ is naturally an $m$-dimensional  complex Lie group. 
 Let $\CM$ be a connected component of it. 
 A \emph{standard section} on $\CM$ is a family $\{ f_{\xi} \in I(\xi) \}_{\xi \in \CM}$ such that $f_{\xi}|_{K_m}$ does not depend on $\xi \in \CM$. 

When $m>1$, we define
\be
\begin{aligned} \label{minmax'}
& \Omega_\eta^m :=
\left\{
(s, \xi) \in \mathbb{C} \times (\widehat{\bk^\times})^m
\;\middle|\;
\begin{array}{l}
\Re(\xi_1) < \Re(\xi_2) < \cdots < \Re(\xi_m), \\[2pt]
-2\Re(\xi_1) < \Re(s) - \Re(\eta) < 1 - 2\Re(\xi_m)
\end{array}
\right\} \\
\subset\ & 
\widehat\Omega_\eta^m :=
\left\{
(s, \xi) \in \mathbb{C} \times (\widehat{\bk^\times})^m
\;\middle|\;
\begin{array}{l}
\Re(\xi_1) < \Re(\xi_2) < \cdots < \Re(\xi_m), \\[2pt]
-2\Re(\xi_1) < \Re(s) - \Re(\eta) < 1 - 2\Re(\xi_{m-1})
\end{array}
\right\} 
\end{aligned}
\ee
where $\xi = (\xi_1,\xi_2,\dots,\xi_m)$ as before. When $m=0$ or $1$, define 
\[
   \Omega_\eta^m:=\widehat\Omega_\eta^m:=\mathbb{C} \times (\widehat{\bk^\times})^m
\]
Note that in all cases $\Omega_\eta^m$ intersects $\BC\times \CM$.  

Recall  the local gamma factor  $\gamma(s, \omega, \psi)$ defined in Tate's thesis, for each character $\omega: \bk^\times\rightarrow \BC^\times$  (see \cite{T50, K03}). It is a meromorphic function of $s\in \BC$.
The relevant analytic properties of the open orbit integrals $\Lambda_{\mathrm{JS}}(s, f, \phi, \varphi_m^{-1})$ are established in the following theorem.

\begin{thml}
\label{thm:FE'_m}
Let $\phi \in \mathcal{S}(\bk^n)$.
\begin{enumerate}
\item For $(s, \xi) \in \widehat\Omega_\eta^m$ and $f \in I(\xi)$, the integral $\Lambda_{\mathrm{JS}}(s,f,\phi,\varphi_m^{-1})$ in \eqref{openint} converges absolutely. If moreover $(s,\xi)\in \Omega_{\eta}^m$, then 
\begin{equation} \label{eq:FE'}
\Lambda_{\mathrm{JS}}(1-s, \sigma'_m \cdot \widetilde{f}, \widehat{\phi}, \varphi_m)
=
\eta(-1)^{mn}\cdot 
\left(\prod_{i=1}^n \gamma(s, \xi_i \xi_{m+1-i} \eta^{-1}, \psi)\right)
\cdot \Lambda_{\mathrm{JS}}(s,f,\phi,\varphi_m^{-1}),
\end{equation}
where $\wt{f}$ is defied in \eqref{wtf}, $\sigma_m'$ is defined in \eqref{taum}, and 
\[
\widehat{\phi} :=
\begin{cases}
\mathcal{F}_\psi(\phi), & \text{if } m \text{ is even},\\
\mathcal{F}_{\overline{\psi}}(\phi), & \text{if } m \text{ is odd}.
\end{cases}
\]

\item Let $\{ f_{\xi} \in I(\xi) \}_{\xi\in \CM}$ be a standard section on a connected component $\CM$ of $(\widehat{\bk^\times})^m$. Then the function
\[
\widehat\Omega_\eta^m \cap (\mathbb{C} \times \CM) \to \mathbb{C},
\quad (s, \xi) \mapsto \Lambda_{\mathrm{JS}}(s, f_\xi, \phi, \varphi_m^{-1})
\]
has a meromorphic continuation to $\mathbb{C} \times \CM$. 
\end{enumerate}
\end{thml}

Note that the condition $(s,\xi) \in \widehat\Omega^m_\eta$ implies that the value $\gamma(s, \xi_i \xi_{m+1-i} \eta^{-1}, \psi)$ in \eqref{eq:FE'} is a well-defined nonzero complex number. 

\begin{remarkl} \label{rmk:Omega}
 It is easy to see that \[
 \Omega_{\eta^{-1}}^m= \{ (1-s, \widetilde{\xi})  \mid (s,\xi) \in \Omega_{\eta}^m \}.
 \]
 Thus the first assertion in Theorem~\ref{thm:FE'_m} (1) implies that the defining integral of $\Lambda_{\mathrm{JS}}(1-s, \sigma'_m \cdot \widetilde{f}, \widehat{\phi}, \varphi_m)$ also converges absolutely when $(s,\xi) \in \Omega_\eta^m$.
\end{remarkl}

We obtain a relation between the Jacquet-Shalika integral $\oZ_{\mathrm{JS}}(s, W_f, \phi, \varphi_m^{-1})$ and the open orbit integral $\Lambda_{\mathrm{JS}}(s, f, \phi, \varphi_m^{-1})$.

\begin{thml}
\label{thm:MF_m}
For $(s,\xi) \in \widehat\Omega^m_\eta$ given in \eqref{minmax'}, $f \in I(\xi)$, and $\phi \in \mathcal{S}(\bk^n)$, it holds that 
\be\label{mfactor}
\Lambda_{\mathrm{JS}}(s,f,\phi,\varphi_m^{-1})
=
\left(\prod_{1 \leq i < j \leq m-i} \gamma(s, \xi_i \xi_j \eta^{-1}, \psi)\right)
\cdot
\oZ_{\mathrm{JS}}(s, W_f, \phi, \varphi_m^{-1}).
\ee
\end{thml}

As before, the condition $(s,\xi) \in \widehat\Omega^m_\eta$ implies that the value  $\gamma(s, \xi_i \xi_j \eta^{-1}, \psi)$ in \eqref{mfactor} is a well-defined nonzero complex number.  Also 
note that by Theorem~\ref{thm:FE_m} (1), the  integral $\oZ_{\mathrm{JS}}(s, W_f, \phi, \varphi_m^{-1})$ in \eqref{mfactor} converges absolutely since $(s,\xi) \in \widehat\Omega^m_\eta$.

\begin{remarkl}
The work of Beuzart-Plessis (\cite{BP21}) on the archimedean theory of local Asai--Rankin--Selberg integrals uses a global method. By choosing an auxiliary split place for a quadratic extension of number fields, he reduces the problem to the known Rankin--Selberg case (\cite{JPSS83, J09}). His global method is not applicable to the Jacquet--Shalika case and relies on the comparison between the Langlands--Shahidi local factors and the Artin local factors. In contrast, our approach is purely local, and the result on modifying factors has important applications to the archimedean period relations in \cite{JLST26}.
\end{remarkl}

\section{Some properties of Jacquet-Shalika Integrals}\label{sec:analytic}


\subsection{Preliminaries on Whittaker functions} \label{sec:WF}

For preparation, we briefly recall some general results from \cite{BP21}. Let $G$ be a connected reductive linear algebraic group over a local field $\bk$ of characteristic zero. Denote by $A_G$ the largest central split torus in $G$ and by $X^*(G)$ the group of algebraic characters of $G$ defined over $\bk$. Put
\[
\mathcal{A}_G^* := X^*(G) \otimes \mathbb{R} = X^*(A_G) \otimes \mathbb{R}
\quad\text{and}\quad
\mathcal{A}_{G,\mathbb{C}}^* := X^*(G) \otimes \mathbb{C} = X^*(A_G) \otimes \mathbb{C}.
\]
We identify $\mathcal{A}_{G,\mathbb{C}}^*$ with the group of unramified characters of $G(\bk)$ by viewing $\lambda = \sum_{i=1}^k \beta_i \otimes s_i \in X^*(G) \otimes \mathbb{C}$ ($k \ge 0$) as the unramified character
\[
g \in G(\bk) \mapsto g^\lambda := \prod_{i=1}^k |\beta_i(g)|_{\bk}^{s_i} \in \mathbb{C}^\times.
\]

Assume that $G$ is quasi-split and fix a Borel subgroup $B=TN$ of $G$, where $T$ is a Levi factor and $N$ is the unipotent radical.  Denote by $\delta_B$ the modular character of $B$. Fix a maximal compact subgroup $K$ of $G(\bk)$ such that the Iwasawa decomposition $G(\bk) = B(\bk)K$ holds.

Let $\Delta \subset X^*(A_T)$ be the set of simple roots of $A_T$ determined by $N$. As usual, for $\alpha \in \Delta$, denote by $\alpha^\vee$ the corresponding simple coroot. Define the closed negative Weyl chamber
\[
\overline{(\mathcal{A}_T^*)^+} :=
\{\, \lambda \in \mathcal{A}_T^* \mid \langle \lambda, \alpha^\vee \rangle \le 0 \text{ for all } \alpha \in \Delta \,\}.
\]
Let $W^G := N_G(T)/T$ be the Weyl group of the pair $(G,T)$, where $N_G(T)$ is the normalizer of $T$ in $G$.  For each $\lambda \in \mathcal{A}_T^*$, denote by $|\lambda|$ the unique element in $(W^G \cdot \lambda) \cap \overline{(\mathcal{A}_T^*)^+}$. Define a partial order $\prec$ on $\mathcal{A}_T^*$ by \footnote{The partial order here is the inverse of the one in 
\cite[Section 2.1]{BP21}. This adjustment is necessary for \cite[Proposition 2.6.1]{BP21} to hold and does not affect the rest of that paper.} 
\[
\lambda \prec \mu
\quad\text{if and only if}\quad
\lambda - \mu = \sum_{\alpha \in \Delta} x_\alpha \alpha
\text{ where } x_\alpha > 0 \text{ for every } \alpha \in \Delta.
\]

Fix an algebraic homomorphism $\imath: G \rightarrow \SL_k(\bk)$ ($k \ge 1$) with kernel $A_G$, and define the log-norm
\begin{equation} \label{log-norm}
\bar{\sigma}(g) :=
\sup \bigl( \{1\} \cup \{\, \log \abs{\imath(g)_{i,j}}_{\bk} \mid i,j = 1,2,\dots, k \,\} \bigr),
\quad g \in G.
\end{equation}
Here $\imath(g)_{i,j}$ denotes the $(i,j)$-entry of the matrix $\imath(g)$.

Let $\psi_N$ be a generic unitary character of $N(\bk)$. For every $\lambda \in \mathcal{A}_T^*$, let $\mathcal{C}_\lambda(N\backslash G, \psi_N)$ be the LF space of Whittaker functions on $G(\bk)$ defined in \cite[Section 2.5]{BP21}, which depends only on $|\lambda|$. Its precise definition will not be recalled here.

We need the following estimate.

\begin{lem}[Lemma 2.5.1 of \cite{BP21}] \label{whit-est}
Let $\lambda \in \mathcal{A}_T^*$. For any $R, d > 0$, there exists a continuous semi-norm $p_{R,d}$ on $\mathcal{C}_\lambda(N\backslash G, \psi_N)$ such that
\[
|W(tk)|
\le
p_{R,d}(W)
\left( \prod_{\alpha \in \Delta} (1 + t^\alpha)^{-R} \right)
\delta_B(t)^{1/2} t^{|\lambda|} \bar{\sigma}(t)^{-d}
\]
for every $W \in \mathcal{C}_\lambda(N\backslash G, \psi_N)$, $t \in T(\bk)$, and $k \in K$.
\end{lem}

Let $\overline{P} $ be a parabolic subgroup of $G$ containing the opposite Borel subgroup $\overline{B}$, and let $M$ be the Levi factor of $\overline{P}$ containing $T$.
The restriction maps $X^*(G)\to X^*(M) \to X^*(T)$ induce linear embeddings 
\be\label{inj}
\mathcal{A}_{G,\BC}^*\hookrightarrow \mathcal{A}_{M,\BC}^* \hookrightarrow \mathcal{A}_{T,\BC}^*,
\ee
and the restriction maps $X^*(A_T)\to X^*(A_M) \to X^*(A_G)$ induce surjective linear maps 
\be\label{surj}
\mathcal{A}_{T,\BC}^*\twoheadrightarrow \mathcal{A}_{M,\BC}^* \twoheadrightarrow \mathcal{A}_{G,\BC}^*.
\ee
In view of \eqref{inj}, we view both $\mathcal{A}_{G,\BC}^*$ and $ \mathcal{A}_{M,\BC}^*$ as subspaces of $\mathcal{A}_{T,\BC}^*$.  The kernel of the second  map in \eqref{surj} is denoted by $(\mathcal{A}_{M,\mathbb{C}}^G)^*$, and the kernel of the composition of \eqref{surj} is denoted by $(\mathcal{A}_{T,\mathbb{C}}^G)^*$ . For every $\mu \in (\mathcal{A}^G_T)^*$ (the obvious real form of $(\mathcal{A}_{T,\mathbb{C}}^G)^*$), denote
\[
\mathcal{U}[\prec \mu] :=
\{\, \lambda \in (\mathcal{A}^G_{M,\mathbb{C}})^* \mid \abs{\Re(\lambda)} \prec \mu \,\}.
\]
Here $\Re(\lambda)\in (\mathcal{A}^G_M)^*$ denote the real part of $\lambda$.

Fix $\tau \in \operatorname{Irr}_{\mathrm{temp}}(M)$. For $\lambda \in \mathcal{A}_{M,\mathbb{C}}^*$, denote by $\tau_\lambda$ the unramified twist of $\tau$ by $\lambda$. Put
\[
\pi_\lambda := \operatorname{Ind}^{G(\bk)}_{\overline{P}(\bk)}(\tau_\lambda)
\quad\text{(normalized smooth induction)}.
\]
Suppose that we are given a nonzero Whittaker linear functional
\[
0 \neq \nu_{\tau,\psi} \in \operatorname{Hom}_{N(\bk) \cap M(\bk)}(\tau, \psi_N|_{N(\bk) \cap M(\bk)}),
\]
which induces a Whittaker linear functional
\[
J_\lambda \in \operatorname{Hom}_{N(\bk)}(\pi_\lambda, \psi_N)
\]
as in \eqref{whittaker1}. Similar to \eqref{whittaker2}, we have a space
\[
\mathcal{W}(\pi_\lambda,\psi_N) :=
\{\, W_\varphi \mid \varphi \in \pi_\lambda \,\}
\subset \operatorname{Ind}_{N(\bk)}^{G(\bk)} \psi_N
\qquad(\text{smooth induction})
\]
of Whittaker functions, equipped with the quotient topology, where $W_\varphi(g) = \langle J_\lambda, g \cdot \varphi \rangle$ for all $g \in G$.

We recall the following result from Proposition 2.6.1 and Corollary 2.7.1 in \cite{BP21}.

\begin{prpl} \label{whit-incl}
Let the notation be as above.
\begin{enumerate}
    \item For all $\lambda \in \mathcal{A}_{M,\mathbb{C}}^*$ and $\mu \in \mathcal{A}_T^*$ such that $|\Re(\lambda)| \prec \mu$,
    \begin{equation}\label{WsubsetC}
    \mathcal{W}(\pi_\lambda,\psi_N) \subset \mathcal{C}_\mu(N\backslash G, \psi_N),
    \end{equation}
    and the inclusion map is continuous.

    \item Let $\mu \in (\mathcal{A}^G_T)^*$. For every analytic section $\{ \varphi_\lambda \in \pi_\lambda \}_{\lambda \in \mathcal{U}[\prec \mu]}$ (see \cite[Section 2.3]{BP21}), the map
    \[
    \mathcal{U}[\prec \mu] \to \mathcal{C}_\mu(N\backslash G, \psi_N),
    \quad \lambda \mapsto W_{\varphi_\lambda},
    \]
    is analytic.
\end{enumerate}
\end{prpl}

\subsection{Jacquet-Shalika integrals revisited} \label{sec2.2}

Now we retain the notation of the Introduction. 
We recall the explicit formulation of the local Jacquet-Shalika integrals following \cite{JS90, CM15}.

Since the element $\sigma'_m$ given by \eqref{taum} is fixed by the MVW involution $h\mapsto {}^t h^{-1}$ on $G_m=\GL_m(\bk)$, the involutions ${\rm Ad}(\sigma'_m)$ and the MVW involution commute. We introduce the following involution:
\be \label{inv}
G_m \to G_m,\qquad h\mapsto \widehat h:= \sigma'_m \, {}^t h^{-1} \sigma'_m.
\ee
One checks immediately that the Shalika subgroup $S_m$ is stable under this involution.

For $k\in \BN$, let $\frak q_k$ denotes the space of upper triangular matrices in $M_k$.
Recall the representation $R_{\varphi_m}$ of $S_m$ defined in Section~\ref{sec1.2.2}. When $m=2n$ is even, as in \cite{JS90}, the Jacquet-Shalika integral in \eqref{JSint} can be written explicitly as
\be \label{JSeven}
\begin{aligned}
\oZ_{\rm JS}(s, W, \phi,\varphi_{2n}^{-1})
&=
\int_{N_n \bsl G_n}
\int_{\frak q_n \bsl M_n}
W\!\left(\sigma_{2n}\begin{bmatrix} g & Xg \\ 0 & g \end{bmatrix}\right)
\overline\psi({\rm tr}\, X) \,dX \\
&\qquad\qquad\qquad\qquad\qquad\qquad\qquad
\phi(e_n g)\,\eta^{-1}(g)\, |g|_\bk^s \,d g.
\end{aligned}
\ee

We record the following result for later use.

\begin{prpl} \label{prop:Reven}
Suppose that $m=2n$. For $\phi \in \CS(\bk^n)$ and $h\in S_{2n}$, it holds that 
\[
R_{\varphi_{2n}}(\widehat h)\bigl( \CF_{\psi}(\phi) \bigr)
=
|h|_{\bk}^{1/2} \, \CF_{\psi}\bigl( R_{\varphi_{2n}^{-1}}(h)\phi \bigr),
\]
where $\widehat h$ is defined in \eqref{inv}.
\end{prpl}

\begin{proof}
As before, write $h = \begin{bmatrix} g & Xg \\ & g \end{bmatrix}$. Then
$
\widehat h =
\begin{bmatrix} {}^t\!g^{-1} & -{}^t\!X \, {}^t\!g^{-1} \\ & {}^t\!g^{-1} \end{bmatrix}.
$
By the definition of the representation $R_{\varphi_{2n}}$ of $S_{2n}$ on $\CS(\bk^n)$ as in \eqref{Seven}, we have 
\[
R_{\varphi_{2n}}(\widehat h)\bigl( \CF_{\psi}(\phi) \bigr)(v)=\varphi_{2n}(\wh{h})\CF_\psi(\phi)(v\cdot\wh{h}). 
\]
By \eqref{Sact} and \eqref{FT}, we have $v\cdot\wh{h}=v{^tg^{-1}}$ and 
\begin{eqnarray*}
\CF_\psi(\phi)(v \cdot \widehat h)
&=&\int_{\bk^n} \phi(x)\,\psi(x\, {^t\!(v \, {}^t\!g^{-1})}) \,d x\\
&=&\int_{\bk^n} \phi(x)\,\psi(v \, {}^t\!g^{-1} \, {}^t\!x) \,d x
=\int_{\bk^n} \phi(x)\,\psi(v \, {}^t\!(xg^{-1}) \,d x.
\end{eqnarray*}
By changing variable: $xg^{-1}\mapsto x$, the last integral is equal to
\[
\CF_\psi(\phi)(v \cdot \widehat h)
=|g|_\bk \int_{\bk^n} \phi(xg)\,\psi(v \, {}^t\!x) \,d x
=|h|_\bk^{1/2}\varphi_{2n}(h)\CF_\psi(R_{\varphi_{2n}^{-1}}(h)(\phi))(v),
\]
for all $v\in\bk^n$. Since $\varphi_{2n}(\widehat h)=\varphi_{2n}^{-1}(h)$, we obtain the identity. 
\end{proof}

Next we elaborate on the odd case. Recall from \eqref{Sodd} the representation
\[
R_{\varphi_{2n+1}} := {\rm ind}^{S_{2n+1}}_{S_{2n+1}\cap P_{2n+1}}\varphi_{2n+1}
\]
(the unnormalized Schwartz induction). We view $\bk^n$ as a subgroup of $S_{2n+1}$ via the embedding
\be\label{embkn}
x\in \bk^n \longmapsto
\begin{bmatrix}
1_n & & \\
0 & 1_n & \\
0 & x & 1
\end{bmatrix}
\in S_{2n+1}.
\ee
The following is a variant of Propositions 3.1 and 3.2 in \cite{CM15}.

\begin{prpl}\label{prop:Rodd}
{\rm (1)} Via restriction through the embedding \eqref{embkn}, the representation $R_{\varphi_{2n+1}}$ is realized on $\CS(\bk^n)$ by
\[
\begin{aligned}
& \left( R_{\varphi_{2n+1}}\!\left(
\begin{bmatrix} g & \\ & g \\ & & 1 \end{bmatrix}
\right)\phi \right)(v)
= \eta(g)\,\phi(vg),\\[4pt]
& \left( R_{\varphi_{2n+1}}\!\left(
\begin{bmatrix} 1_n & X & 0 \\ & 1_n & 0 \\ & & 1 \end{bmatrix}
\right)\phi \right)(v)
= \psi({\rm tr}\, X)\,\phi(v),\\[4pt]
& \left( R_{\varphi_{2n+1}}\!\left(
\begin{bmatrix} 1_n & 0 & y \\ & 1_n & 0 \\ & & 1 \end{bmatrix}
\right)\phi \right)(v)
= \psi(-v y)\,\phi(v),\\[4pt]
& \left( R_{\varphi_{2n+1}}\!\left(
\begin{bmatrix} 1_n & & \\ 0 & 1_n & \\ 0 & x & 1 \end{bmatrix}
\right)\phi \right)(v)
= \phi(v+x),
\end{aligned}
\]
where $\phi\in\CS(\bk^n)$, $g\in G_n$, $X\in M_n$, $y\in \bk^{n\times 1}$, and $x,v\in \bk^{1\times n}$.

{\rm (2)} For $\phi\in\CS(\bk^n)$ and $h\in S_{2n+1}$, it holds that 
\[
R_{\varphi_{2n+1}}(\widehat h)\bigl( \CF_{\overline\psi}(\phi) \bigr)
=
|h|_\bk^{1/2} \, \CF_{\overline\psi}\bigl( R_{\varphi_{2n+1}^{-1}}(h)\phi \bigr),
\]
where $\widehat h$ is given by \eqref{inv}.
\end{prpl}

When $m=2n+1$ is odd, as in \cite{CM15}, the Jacquet-Shalika integral in \eqref{JSint} can be written explicitly as
\be \label{JSodd}
\begin{aligned}
\oZ_{\rm JS}(s, W, \phi,\varphi_{2n+1}^{-1})
&=
\int_{N_n \bsl G_n}
\int_{\frak q_n \bsl M_n}
\int_{\bk^n}
W\!\left(\sigma_{2n+1}
\begin{bmatrix} g & Xg & 0 \\ & g & 0 \\ & x & 1 \end{bmatrix}
\right)
\phi(x) \,\od\! x \\
&\qquad\qquad\qquad\qquad\qquad\qquad
\overline\psi({\rm tr}\, X) \,\od\!X \,
\eta^{-1}(g)\, |g|_\bk^{s-1} \,\od\! g.
\end{aligned}
\ee

To ease notation, for a subgroup $\CG$ of $G_n$, set
\be \label{dag}
\CG^\dag:=\{ g^\dag : g\in\CG \}\subset S_{2n}
\qquad\text{and}\qquad
\CG^{\ddag}:=\{ g^{\ddag} : g\in\CG \}\subset S_{2n+1},
\ee
where for $g\in G_n$ we write
\[
g^\dag := \begin{bmatrix} g & \\ & g \end{bmatrix} \in S_{2n}
\qquad\text{and}\qquad
g^{\ddag} := \begin{bmatrix} g & & \\ & g & \\ & & 1 \end{bmatrix} \in S_{2n+1}.
\]

\subsection{Convergence and continuity}\label{ssec:CC}

We apply the discussion in Section~\ref{sec:WF} to the case when $(G(\bk),B(\bk),T(\bk))=(G_m,B_m,A_m)$.  Then $\CA_T^*=\BR^m$, and the closed negative Weyl chamber is
\[
\overline{(\CA_T^*)^+}=
\left\{ \lambda = (\lambda_1,\ldots, \lambda_m)\in \BR^{m} \mid \lambda_1\leq \cdots\leq \lambda_m \right\}.
\]
For $\lambda = (\lambda_1,\ldots, \lambda_m)\in \CA_T^*$, we have that $|\lambda| = (\lambda_{w(1)},\ldots, \lambda_{w(m)})$ for any permutation $w\in \frak S_m$ such that $\lambda_{w(1)}\leq \cdots \leq \lambda_{w(m)}$. 
Similar to \eqref{minmax}, when $m>0$  we put
\[
\min \lambda := \min_{i=1,2,\dots, m} \lambda_i.
\]

We now introduce some additional notation that will be used later, for each  $k\in \BN$:
\begin{itemize}
\item 
$\delta_k$ denotes the modular character of 
$B_k = A_k N_k$; 
\item 
 $\bar{\frak v}_k$ denotes the space of strictly lower triangular matrices in $M_k$, so that $M_k = \frak q_k \oplus \bar{\frak v}_k$;

\item 
$U_k$ denotes the unipotent radical of the mirabolic subgroup $P_k$,  and $\overline{U}_k := {}^t U_k$ denotes the opposite unipotent radical;
\item 
$Z_k$ denotes the center of $G_k$.
\end{itemize}
For every subset $I\subset \R$, set 
\[
\CH_I:=\{x\in \BC\mid \Re(x) \in I\}.
\]
 A vertical strip is a subset of \(\BC\) of the form \(\CH_{[a,b]}\), where \(a,b\in \R\) with \(a<b\).

For later use, we recall the following well-known results (\cf~\cite[Theorem~2.27]{F99}), which are easy consequences of the dominated convergence theorem (cf.~\cite[Theorem~1.34]{Ru87}) and Morera's theorem (cf.~\cite[Theorem~10.17]{Ru87}).

\begin{lem}\label{lem:Lebesgue}
    Let $X$ be a measure space and let $\Omega$ be a smooth manifold. 
    Suppose that $f: X \times \Omega \to \BC$ is a function such that 
    $f(\cdot, z) \in L^1(X)$ for every $z \in \Omega$. 
    Define a function $F$ on $\Omega$ by 
    \[
    F(z) := \int_X f(x, z) \, \od\! x.
    \]
    
    \begin{enumerate}
        \item Assume that $f(x, \cdot)$ is smooth for almost every $x \in X$ (i.e., for all $x\in X$ outside a measure-zero set), 
        and that for every nonempty compact subset $K \subset \Omega$ and every 
        differential operator $D$ on $\Omega$, there exists a real-valued function 
        $f_{K, D} \in L^1(X)$ such that, for every $z \in K$,
        \[
        |Df(x, z)| \leq f_{K, D}(x) \quad \text{for almost every } x \in X.
        \]
       Then $F$ is smooth, and for every differential 
        operator $D$ on $\Omega$, it holds that 
        \[
        DF(z) = \int_X Df(x, z) \, \od\!x, \qquad z \in \Omega.
        \]
        
        \item Assume that $\Omega$ is a complex manifold, that $f(x, \cdot)$ is holomorphic 
        for almost every $x \in X$, and that for every nonempty compact subset $K \subset \Omega$, 
        there exists a real-valued function $f_K \in L^1(X)$ such that, for every $z \in K$,
        \[
        |f(x, z)| \leq f_K(x) \quad \text{for almost every } x \in X.
        \]
        Then $F$ is holomorphic, and for every holomorphic 
        differential operator $D$ on $\Omega$,
        \[
        DF(z) = \int_X Df(x, z) \, \od\!x, \qquad z \in \Omega,
        \]
        i.e., the integral on the right-hand side is absolutely convergent and equals 
        the left-hand side.
    \end{enumerate}
\end{lem}

In view of Proposition~\ref{whit-incl}, we begin with the following result.

\begin{prpl} \label{prop:conv0} Suppose $m>1$ and let $\mu\in\CA_T^*$.
\begin{enumerate}
    \item 
    The map
    \be\label{zetaconv}
    \begin{array}{rcl}
    \CH_{(\Re(\eta) -2\min\mu,\infty)}
    \times \CC_\mu(N_m \bsl G_m, \psi_m)\times \CS(\bk^n)
    & \rightarrow & \BC,\\[2pt]
    (s,W,\phi) & \mapsto & \oZ_{\rm JS}(s, W, \phi, \varphi_m^{-1})
    \end{array}
    \ee
    is well defined by absolutely convergent integrals.
    
    \item The map \eqref{zetaconv} is continuous, holomorphic in the first variable, and linear in the last two variables.
    
    \item The map \eqref{zetaconv} is bounded on vertical strips in the following sense: 
    for all vertical strips $\CV \subset \CH_{(\Re(\eta)-2\min\mu, \infty)}$, there exist continuous semi-norms $p_{\CV}$ on $\CC_\mu(N_m \bsl G_m, \psi_m)$ and $q_{\CV}$ on $\CS(\bk^n)$ such that 
    \[
    |\oZ_{\rm JS}(s, W, \phi, \varphi_m^{-1})| \leq p_{\CV}(W) q_{\CV}(\phi)
    \]
    for all $(s,W,\phi)\in \CV\times \CC_\mu(N_m \bsl G_m, \psi_m)\times \CS( \bk^n)$. 
\end{enumerate}
\end{prpl}

\begin{proof} We only prove the case where $m=2n$ is even. The odd case can be proved similarly with suitable modifications using the proof of 
Proposition 3 in \cite[Section 9]{JS90}, which will be omitted. 

By unramified twists, we assume without loss of generality  that $\eta$ is unitary so that $\Re(\eta)=0$. Write $\CV= \CH_{[C,D]}$ with 
\[
- 2\min\mu < C <D.
\]
By the Iwasawa decomposition $G_n=N_n A_n K_n$, we need to estimate the integral 
\[
\int_{A_n \times \bar{\frak v}_n\times K_n}  
\left| W \!\left( \sigma_{2n} \begin{bmatrix} 1_n & X \\ 0 & 1_n\end{bmatrix} (ak)^\dag \right) \phi(e_n ak) \right|  
\cdot | a|_\bk^s \, \delta_n(a)^{-1} \od\!a \od\! X\od\! k,
\]
where $(s, W,\phi)\in [C, D]\times \CC_\mu(N_m \bsl G_m, \psi_m)\times \CS( \bk^n)$.

For $X\in \bar{\frak v}_n$, introduce an element
\be \label{uX}
u_X := \sigma_{2n} \begin{bmatrix} 1_n & X \\ 0 & 1_n \end{bmatrix}\sigma_{2n}^{-1}.
\ee
By a change of variable $X\mapsto aXa^{-1}$, the above integral can be written as 
\[
\begin{aligned}
& \int_{A_n \times \bar{\frak v}_n\times K_n}  
\left|W\left(\sigma_{2n} a^\dag \begin{bmatrix}  1_n & X \\ 0 & 1_n\end{bmatrix} k^\dag \right)  \phi(e_n ak) \right| \cdot | a|^s_\bk \, \delta_n(a)^{-2}  \od\!a \od\! X \od\! k\\
= & \int_{A_n \times \bar{\frak v}_n\times K_n}  
|W(\tilde{a} u_X \sigma_{2n} k^\dag )  \phi(e_n ak) | \cdot | a|^s_\bk \, \delta_n(a)^{-2}  \od\!a \od\! X \od\! k,
\end{aligned}
\]
where for $a=\diag (a_1, a_2, \ldots, a_n)\in A_n$ we set
\[
\tilde{a} := \diag(a_1, a_1, a_2, a_2, \ldots, a_n, a_n)\in A_{2n}
\]
so that 
\[
\sigma_{2n} a^\dag=\tilde a \sigma_{2n}.
\]
Following 
the Iwasawa decomposition, we write 
\[
u_X = n_X t_X k_X \in N_{2n} A_{2n} K_{2n}\quad\textrm{and}\quad  
t_X =\diag(t_1,\ldots, t_{2n})\in A_{2n}.
\]
The above integral is 
\[
\int_{A_n \times \bar{\frak v}_n\times K_n}  
|W(\tilde a \, t_X k_X  \sigma_{2n} k^\dag ) \phi(e_n ak)|   \cdot  | a|^s_\bk \,  \delta_n(a)^{-2} \od\!a \od\! X \od\! k.
\]

For every $R>0$, define the following continuous semi-norm on $\CS(\bk^n)$:
\[
q_R(\phi):=\sup_{a=\diag (a_1, a_2, \ldots, a_n)\in A_n, k\in K_n} (1+|a_n|_\bk)^R |\phi(e_n ak)| <\infty. 
\]
It is straightforward to verify that 
$\delta_{2n}(\tilde a)^{1/2} = \delta_n(a)^2$. 
Thus by Lemma~\ref{whit-est}, the above integral is bounded by
\[
q_R(\phi)\cdot q_{R,1}(W)\cdot  \int_{A_n \times \bar{\frak v}_n} \left( \prod^{2n-1}_{i=1} (1 + (\tilde a t_X)^{\alpha_i})^{-R} \right)
\delta_{2n}( t_X)^{1/2} (\tilde a t_X)^{|\mu|} (1+|a_n|_\bk)^{-R} |a|^s_\bk  \od\!a \od\! X ,
\]
where $q_{R,1}$ is a continuous semi-norm on $\mathcal{C}_\mu(N_m\backslash G_m, \psi_m)$ as in  Lemma \ref{whit-est} (for $d=1$), and 
$\alpha_1, \alpha_2,\ldots, \alpha_{2n-1}\in X^*(A_{2n})$ are the simple roots with respect to $N_{2n}$.

Now we are reduced  to estimating 
\[
\begin{aligned}
\int_{A_n\times \bar{\frak v}_n} & \prod^n_{i=1}  \left(1+\left| \frac{t_{2i-1}}{t_{2i}} \right|_\bk\right)^{-R} \cdot \prod^{n-1}_{i=1} \left(1+\left| \frac{a_i t_{2i}}{a_{i+1}t_{2i+1}}\right|_\bk\right)^{-R} \cdot \delta_{2n}(t_X)^{1/2}  t_X^{|\mu|}\\
& \qquad\qquad\qquad\qquad \cdot  
(1+|a_n|_\bk)^{-R}  \prod^n_{i=1} |a_i|_\bk^{s+|\mu|_{2i-1}+|\mu|_{2i}} \od\! a \od\! X,
\end{aligned}
\]
where we write $|\mu| = (|\mu|_1,\ldots, |\mu|_{2n})$. Making the change of variables 
\[
a_i \mapsto a_i t_i',\quad \text{where}\quad t_i':=\frac{t_{2i+1} t_{2i+3}\cdots t_{2n-1}}{t_{2i} t_{2i+2}\cdots t_{2n-2}},\quad i=1,2,\ldots, n-1,
\]
we are further reduced to estimating the product of the integral
\be \label{int1}
\int_{\bar{\frak v}_n} \left( \prod^n_{i=1} \left(1+\left| \frac{t_{2i-1}}{t_{2i}} \right|_\bk\right)^{-R}\right)\cdot \mu_s(t_X)\od\!X
\ee
where 
\[
\mu_s(t_X) : = \left(\prod^{n-1}_{i=1} |t_i'|_\bk^{s+|\mu|_{2i-1}+|\mu|_{2i}}\right)\cdot \delta_{2n}(t_X)^{1/2}  t_X^{|\mu|},
\]
and the integral
\be \label{int2}
\int_{A_n} \left( \prod^{n-1}_{i=1} \left(1+\left| \frac{a_i }{a_{i+1}}\right|_\bk\right)^{-R}\right)\cdot 
(1+|a_n|_\bk)^{-R}\cdot 
 \prod^n_{i=1} |a_i|_\bk^{s+|\mu|_{2i-1}+|\mu|_{2i}} \od\! a.
\ee
  
By Propositions 4 and 5 in \cite[Section 5]{JS90}, there exists $\alpha>0$ such that 
\[
\prod^n_{i=1} \left(1+\left| \frac{t_{2i-1}}{t_{2i}} \right|_\bk\right) \geq \prod^n_{i=1} | t_{2i-1}|_\bk \geq e^{\alpha\cdot \bar\sigma(u_X)}\qquad \textrm{for all }X\in \frak{\bar v}_n,
\]
where $\bar\sigma$ is the log-norm \eqref{log-norm}.
Since $\mu_s(t_X)$ is uniformly of polynomial growth in $X$ for $s\in [C, D]$, 
we can assume that $R$ is sufficiently large (depending on $\alpha$, $\mu$, and $[C,D]$) such that the integrand in \eqref{int1} is bounded by an integrable function that 
does not depend on $s\in [C,D]$.

We further assume that 
\[
R > n(D+2 \max\mu).
\]
The integral \eqref{int2} can be estimated in the same way as in the proof of \cite[Lemma 3.3.1]{BP21}. By the elementary inequality 
\[
\left(\prod^{n-1}_{i=1} \left(1+\left| \frac{a_i }{a_{i+1}}\right|_\bk\right)^{-R}\right)\cdot 
(1+|a_n|_\bk)^{-R} \leq \prod^n_{i=1} (1+|a_i|_\bk)^{-R/n},
\]
and the convergence of the integral 
\[
\int_{\bk^\times} (1+|x|_\bk)^{-R/n} |x|_\bk^{R'} \od^\times\!x
\]
when $R > n R'>0$, we find that the integrand in \eqref{int2} is bounded by an integrable function that does not depend on $s\in [C, D]$. 
 
This implies parts (1) and (3) of the proposition. The continuity assertion of Part (2) then follows by the Lebesgue dominated convergence theorem, and the holomorphy assertion follows by Lemma~\ref{lem:Lebesgue} (the linearity assertion is obvious).
\end{proof}

The following result establishes the absolute convergence in Theorem~\ref{thm:FE_m} (1). 

\begin{corl}\label{cor:conv} Suppose $m>1$.  Let $\mu\in\CA_T^*$, and let $\pi_\lambda = \Ind^{G_n}_{\overline{P}}(\tau_\lambda)$ be given by \eqref{nt}. 
    \begin{enumerate}
        \item 
        The map
        \be\label{zetaconv2}
        \begin{array}{rcl}
        \CH_{( \Re(\eta)-2\min\mu,\infty)}
        \times \CU[\prec\mu]\times \pi_\tau\times \CS(\bk^n)
        & \rightarrow & \BC,\\[2pt]
        (s,\lambda,f,\phi) & \mapsto & \oZ_{\rm JS}(s, W_{f_\lambda}, \phi, \varphi_m^{-1})
        \end{array}
        \ee
        is well defined by absolutely convergent integrals.
        
        \item The map \eqref{zetaconv2} is continuous, holomorphic in the first two variables, and linear in the last two variables.
        
        \item The map \eqref{zetaconv2} is bounded on vertical strips in the following sense: 
        for all vertical strips $\CV \subset \CH_{(\Re(\eta)-2\min \mu, \infty)}$ and $\lambda\in \CU[\prec\mu]$, there exist continuous semi-norms $p_{\CV,\lambda}$ on $\pi_\tau$ and $q_{\CV}$ on $\CS(\bk^n)$ such that 
        \[
        |\oZ_{\rm JS}(s,  W_{f_\lambda}, \phi, \varphi_m^{-1})| \leq p_{\CV,\lambda}(f) q_{\CV}(\phi)
        \]
        for all $(s,f,\phi)\in \CV\times \pi_\tau\times \CS( \bk^n)$. 	
    \end{enumerate}
\end{corl}
    
\begin{proof}
This follows immediately from Propositions~\ref{whit-incl} and \ref{prop:conv0}.
\end{proof}

\subsection{A non-vanishing result}

We have the following non-vanishing result.

\begin{prpl} \label{prop:nonv}
Let $\pi \in \Irr_{\rm gen}(G_m)$. For every $s_0 \in \BC$, there exist finitely many 
$W_i \in \CW(\pi,\psi_m)$ and $\phi_i\in \CS(\bk^n)$ indexed by $i\in I$, such that the function 
\[
s\mapsto \sum_{i\in I} \oZ_{\rm JS}(s, W_i, \phi_i, \varphi_m^{-1}),
\]
which is defined for $\Re(s)$ sufficiently large, 
has a holomorphic extension to $\BC$ and is non-vanishing at the given $s_0\in\BC$.
\end{prpl}

\begin{proof} 
Again we only give the proof for the case $m=2n$ even; the odd case is similar and is omitted. The argument follows the lines of \cite[Lemma~3.3.3]{BP21}.

Note that $P_n Z_n \overline{U}_n \subset G_n$ is open and dense. By Corollary~\ref{cor:conv}, for $W\in \CW(\pi,\psi_m)$, $\phi\in \CS(\bk^n)$ and $\Re(s)$ sufficiently large, we have the absolutely convergent integral
\[
\begin{aligned}
\oZ_{\rm JS}(s, W, \phi, \varphi_{2n}^{-1})
&= \int_{Z_n\times \overline{U}_n} \int_{N_n\bsl P_n \times \bar{\frak v}_n} 
   W(u_X \sigma_{2n} (pz\bar u)^\dag) \eta^{-1}(p) |p|_\bk^{s-1} \od\!p \od\!X \\
&\qquad\qquad \cdot \phi(e_n z\bar u) \eta^{-1}(z) |z|_\bk^s \od\!z \od\!\bar u \\
&= \int_{Z_n \times \overline{U}_n} \int_{N_n\bsl P_n \times \bar{\frak v}_n} 
   W(u_X \sigma_{2n} (p\bar u)^\dag) \eta^{-1}(p) |p|_\bk^{s-1} \od\!p \od\!X \\
&\qquad\qquad \cdot \phi(e_n z\bar u) \omega_\pi(z^\dag) \eta^{-1}(z) |z|_\bk^s \od\!z \od\!\bar u,
\end{aligned}
\]
where $u_X$ is as in \eqref{uX} and $\omega_\pi$ is the central character of $\pi$.

For $\varphi_Z \in C^\infty_c(Z_n)$ and $\varphi_{\overline{U}}\in C^\infty_c(\overline{U}_n)$, there is a unique 
$\phi = \phi_{\varphi_Z, \varphi_{\overline{U}}} \in C^\infty_c(\bk^n)$ such that 
\[
\phi(e_n z\bar u) = \varphi_Z(z) \varphi_{\overline{U}}(\bar u)
\]
for all $(z, \bar u)\in Z_n\times \overline{U}_n$. By abuse of notation, 
view $\varphi_{\overline{U}}$ as a function on $\overline{U}_n^\dag$. Then for the above $\phi$ and $\Re(s)$ sufficiently large, we have
\[
\begin{aligned}
\oZ_{\rm JS}(s, W, \phi, \varphi_{2n}^{-1})
&= \int_{N_n\bsl P_n \times \bar{\frak v}_n} 
   \bigl(R(\varphi_{\overline{U}})W\bigr)(u_X\sigma_{2n} p^\dag)\eta^{-1}(p) |p|_\bk^{s-1} \od\!p \od\!X \\
&\qquad\qquad \cdot \int_{Z_n}\varphi_Z(z) \omega_\pi(z^\dag) \eta^{-1}(z) |\det z|_\bk^s \od\!z,
\end{aligned}
\]
where $R(\varphi_{\overline{U}})$ denotes the right regular action. The Tate integral
\[
\zeta(s, \varphi_Z):=\int_{Z_n}\varphi_Z(z) \omega_\pi(z^\dag) \eta^{-1}(z) |z|_\bk^s \od\!z
\]
converges absolutely for all $s\in\BC$, and we can choose $\varphi_Z$ such that $\zeta(s_0, \varphi_Z)\neq 0$.

By \cite{GK}, \cite[Proposition 5]{J10} and \cite{K15},  for any $f \in C^\infty_c(N_{2n}\bsl P_{2n}, \psi_{2n})$, there exists 
$W_0\in \CW(\pi,\psi_{2n})$ whose restriction to $P_{2n}$ coincides with $f$. By the Dixmier–Malliavin lemma, there exist finitely many $W_i \in \CW(\pi, \psi_{2n})$ and $\varphi_{\overline{U},i}\in C^\infty_c(\overline{U}_n)$, indexed by $i\in I$, such that $W_0 = \sum_{i\in I}R(\varphi_{\overline{U},i})W_i$. 
Put $\phi_i := \phi_{\varphi_Z, \varphi_{\overline{U},i}}$ for $i\in I$. Then for $\Re(s)$ sufficiently large we have
\[
\begin{aligned}
\sum_{i\in I} \oZ_{\rm JS}(s, W_i, \phi_i, \varphi_{2n}^{-1})
&= \sum_{i\in I} 
   \int_{N_n\bsl P_n\times \bar{\frak v}_n}
   \bigl(R(\varphi_{\overline{U},i})W_i\bigr)(u_X\sigma_{2n} p^\dag)\eta^{-1}(p) |p|_\bk^{s-1} \od\!p \od\!X 
   \cdot \zeta(s, \varphi_Z) \\
&= \int_{N_n\bsl P_n\times \bar{\frak v}_n} 
   W_0(u_X\sigma_{2n} p^\dag)\eta^{-1}(p) |p|_\bk^{s-1} \od\!p \od\!X 
   \cdot \zeta(s, \varphi_Z) \\
&= \int_{N_n\bsl P_n\times \bar{\frak v}_n} 
   f(u_X\sigma_{2n} p^\dag)\eta^{-1}(p) |p|_\bk^{s-1} \od\!p \od\!X 
   \cdot \zeta(s, \varphi_Z),
\end{aligned}
\]
noting that $u_X\sigma_{2n} p^\dag \in P_{2n}$. The above integrals converge absolutely for all $s\in\BC$, uniformly on compacta, hence define a holomorphic function on $\BC$. 

We can choose $f$ such that
\[
\int_{N_n\bsl P_n\times \bar{\frak v}_n} 
f(u_X\sigma_{2n} p^\dag)\eta^{-1}(p) |p|_\bk^{s_0-1} \od\!p \od\!X \neq 0.
\]
Since we have chosen $\varphi_Z$ with $\zeta(s_0,\varphi_Z)\neq 0$, the holomorphic continuation of 
$\sum_{i\in I} \oZ_{\rm JS}(s, W_i, \phi_i, \varphi_{2n}^{-1})$ does not vanish at $s_0$. 
\end{proof}

\section{Proof of Theorem \ref{thm:FE'_m} (1)} \label{sec:FE'}

In this section we prove Theorem \ref{thm:FE'_m} (1). 

\subsection{Absolute convergence} \label{sec:CC}

We first establish the following convergence result.

\begin{leml}\label{lem:CC}
For $(s, \xi) \in \widehat\Omega_\eta^m$ given in \eqref{minmax'} and $f \in I(\xi)$, the integral $\Lambda_{\mathrm{JS}}(s,f,\phi,\varphi_m^{-1})$ in \eqref{openint} converges absolutely.
\end{leml}

\begin{proof}  
If $m=2n$, recall the open orbit integral from \eqref{openint}, which after a change of variable $X\mapsto gXg^{-1}$ can be written explicitly as 
\be \label{openinteven}
\Lambda_{\rm JS}(s, f, \phi, \varphi_{2n}^{-1})
= \int_{G_n}\int_{M_n} 
f\!\left( \begin{bmatrix} g &  gX  \\ 0& w_ng \end{bmatrix}\right)
\phi(v_n g) \, \psi(-{\rm tr}\, X) \, \od\! X \, \eta^{-1}(g) |g|_\bk^{s}\,\od\! g.
\ee

To prove absolute convergence of \eqref{openinteven} for $(s,\xi)\in \widehat\Omega^{2n}_\eta$, we assume without loss of generality that 
$\xi$ takes positive real values (as a character of ($\bk^\times)^m$) and that both 
$f$ and $\phi$ are real-valued and nonnegative. Using the Gindikin–Karpelevich formula (see \cite{GGPS} and \cite{L71}) together with Lemma~\ref{lem:Lebesgue}, we find that 
\begin{equation}\label{inti}
\int_{M_n} f\!\left(\begin{bmatrix} g_1 & g_1 X \\0 & g_2\end{bmatrix}\right) \od\!X, \qquad g_1, g_2\in G_n,
\end{equation}
converges absolutely and defines a smooth function of $g_1,g_2\in G_n$. Moreover, it belongs to
\[
I(\xi^1)\abs{\cdot}_{\bk}^{-n/2}\,\widehat\otimes\, I(\xi^2)\abs{\cdot}_{\bk}^{n/2},
\]
where $\xi^1:=(\xi_1,\dots,\xi_n)$ and $\xi^2:=(\xi_{n+1},\dots,\xi_{2n})$. Recall $v_n=(1,1,\dots,1)\in \bk^n$ from \eqref{vn}. It is readily checked that $(\overline{B}_n, \overline{B}_n w_n, v_n)$ is a base point for the unique open orbit of the diagonal right action of $G_n$ on $\CB_n\times \CB_n\times \bk^n$. Consequently, \cite[Proposition~1.4]{LLSS23} implies the absolute convergence of \eqref{openinteven} on the desired domain.

The odd case $m=2n+1$ is handled similarly, invoking \cite[Proposition~1.4]{LLSS23} and the fact that
\[
\left(\overline{B}_{n},\, \overline{B}_{n+1} \begin{bmatrix} w_n & {}^t v_n \\0 & 1\end{bmatrix} \right)
\]
is a base point of the unique open $G_n$-orbit in $\CB_n\times \CB_{n+1}$. We omit the details.
\end{proof}

By Lemma~\ref{lem:CC}, we have proved the convergence assertion of Theorem~\ref{thm:FE'_m} (1). In the next two subsections, we prove the   functional equation \eqref{eq:FE'} in Theorem~\ref{thm:FE'_m} (1). 

\subsection{Proof of the functional equation \eqref{eq:FE'}: the even case} \label{sec5.2}

Assume in this subsection that $m=2n$. In this case, \eqref{eq:FE'} reduces to
\[
\Lambda_{\rm JS}(1-s, \sigma'_{2n}\cdot  \tilde f, \CF_\psi(\phi), \varphi_{2n})
= \left(\prod^n_{i=1} \gamma(s, \xi_i \xi_{2n+1-i} \eta^{-1}, \psi)\right) \cdot  \Lambda_{\rm JS}(s,f,\phi,\varphi_{2n}^{-1}),
\]
where $s\in \Omega_{\xi,\eta}:= \set{ s \in \BC | (s, \xi) \in \Omega_\eta^m}$. Using the definition and the identity ${}^tz_{2n}^{-1}=z_{2n}$, we obtain
\[
\Lambda_{\rm JS}(1-s, \sigma'_{2n}\cdot \tilde f, \CF_\psi(\phi), \varphi_{2n})
= \int_{S_{2n}} f(w_{2n}z_{2n} {}^th^{-1} \sigma'_{2n})
\left(R_{\varphi_{2n}}(h)\CF_\psi(\phi)\right)(v_n) | h|_\bk^{\frac{1-s}{2}} \od\! h.
\]
A direct calculation gives $w_{2n}z_{2n} \sigma'_{2n} = z_{2n}$. Thus, applying the change of variables $h\mapsto \widehat h$ (see \eqref{inv}) and invoking Proposition~\ref{prop:Reven}, we obtain
\be \label{Lambda1}
\Lambda_{\rm JS}(1-s, \sigma'_{2n}\cdot  \tilde f, \CF_\psi(\phi), \varphi_{2n})
= \int_{S_{2n}} f(z_{2n} h) \CF_\psi(R_{\varphi_{2n}^{-1}}(h)\phi)(v_n) |h|_{\bk}^{\frac{s}{2}}\od\! h. 
\ee

Write  \eqref{Lambda1} as the iterated integral 
$\int_{ A_n^\dag\bsl S_{2n}} \int_{A_n^\dag}$. For $a = \diag(a_1, a_2,\dots, a_n)\in A_n$ and 
$a^\dag =\begin{bmatrix} a \\ & a \end{bmatrix}\in S_{2n}$,  using Proposition~\ref{prop:Reven} again we verify that
\[
\begin{aligned}
& f(z_{2n} a^\dag h) \CF_\psi(R_{\varphi_{2n}^{-1}}(a^\dag h)\phi)(v_n) | a^\dag h|_{\bk}^{\frac{s}{2}} \\
= \,&  f(z_{2n}  h) | h|_\bk^{\frac{s}{2}}\cdot\left(\prod^n_{i=1} (\xi_i \xi_{2n+1-i}\eta^{-1})(a_i) |a_i|_\bk^{s-1} \right)\cdot   \CF_\psi(R_{\varphi_{2n}^{-1}}(h)\phi)(a_1^{-1}, \dots, a_n^{-1}) .
\end{aligned}
\]
By the change of variables $a\mapsto a^{-1}$ and Tate's thesis, we get
\[
\begin{aligned}
& \int_{A_n^\dag}  \left(\prod^n_{i=1} (\xi_i \xi_{2n+1-i}\eta^{-1})(a_i) |a_i|_\bk^{s-1} \right)\cdot   \CF_\psi(R_{\varphi_{2n}^{-1}}(h)\phi)(a_1^{-1}, \dots, a_n^{-1}) \od\! a^\dag \\
= & \prod^n_{i=1} \gamma(s, \xi_i \xi_{2n+1-i}\eta^{-1}, \psi) \cdot\int_{A_n^\dag}\left( \prod^n_{i=1}(\xi_i \xi_{2n+1-i}\eta^{-1})(a_i) |a_i|_\bk^s\right) \cdot \left(R_{\varphi_{2n}^{-1}}(h)\phi\right)(a_1, \dots, a_n) \od\! a^\dag,
\end{aligned}
\]
where both integrals converge absolutely. In view of the preceding equation and the identity
\[
\begin{aligned}
& f(z_{2n} a^\dag h) \left(R_{\varphi_{2n}^{-1}}(a^\dag h)\phi\right)(v_n) | a^\dag h |_{\bk}^{\frac{s}{2}} \\
= \,&  f(z_{2n}  h) | h|_\bk^{\frac{s}{2}}\cdot \left(\prod^n_{i=1} (\xi_i \xi_{2n+1-i}\eta^{-1})(a_i) |a_i|_\bk^{s}\right) \cdot   \left(R_{\varphi_{2n}^{-1}}(h)\phi\right)(a_1, \dots, a_n),
\end{aligned}
\]
we find that \eqref{Lambda1} equals
\[
\begin{aligned}
&  \left(\prod^n_{i=1} \gamma(s, \xi_i \xi_{2n+1-i}\eta^{-1}, \psi)\right)  \cdot \int_{S_{2n}}f(z_{2n} h) \left(R_{\varphi_{2n}^{-1}}(h)\phi\right)(v_n) | h|_\bk^{\frac{s}{2}}\od\! h \\
= & \left( \prod^n_{i=1} \gamma(s, \xi_i \xi_{2n+1-i}\eta^{-1}, \psi)\right) \cdot \Lambda_{\rm JS}(s, f, \phi, \varphi_{2n}^{-1}).
\end{aligned}
\]
This proves \eqref{eq:FE'} in the even case.

\subsection{Proof of the functional equation \eqref{eq:FE'}: the odd case}

Assume in this subsection that $m=2n+1$. Then \eqref{eq:FE'} takes the form
\[
\Lambda_{\rm JS}(1-s, \sigma'_{2n+1}\cdot  \tilde f,  \CF_{\overline\psi}(\phi), \varphi_{2n+1})
=\eta(-1)^n \cdot \left(\prod^n_{i=1} \gamma(s, \xi_i \xi_{2n+2-i} \eta^{-1}, \psi)\right) \cdot  \Lambda_{\rm JS}(s,f,\phi,\varphi_{2n+1}^{-1}),
\]
where $s\in \Omega_{\xi,\eta}:= \set{ s \in \BC \mid (s, \xi) \in \Omega_\eta^m}$.

We begin with the integral expression
\be \label{Lambda2}
\begin{aligned}
& \Lambda_{\rm JS}(1-s, \sigma'_{2n+1}\cdot  \tilde f,  \CF_{\overline\psi}(\phi), \varphi_{2n+1}) \\
  =\, & \int_{S_{2n+1}} f(w_{2n+1} {}^t z_{2n+1}^{-1} {}^t h^{-1} \sigma'_{2n+1})\left(R_{\varphi_{2n+1}}(h) \CF_{\overline\psi}(\phi)\right)(0)| h|_\bk^{\frac{1-s}{2}}\od\!h \\
= \, & \int_{S_{2n+1}} f(z_{2n+1}' \widehat h)\left(R_{\varphi_{2n+1}}(h) \CF_{\overline\psi}(\phi)\right)(0)| h|_\bk^{\frac{1-s}{2}}\od\!h,
\end{aligned}
\ee
where
\be \label{z'}
z_{2n+1}' := w_{2n+1} {}^t\!z_{2n+1}^{-1}\sigma'_{2n+1} = \begin{bmatrix} -v_n & 0& 1 \\ 1_n &0 & 0\\ 0& w_n & 0\end{bmatrix}. 
\ee

Unlike the even case, the computation in the odd case is considerably more involved. We first record a useful factorization of the element $z_{2n+1}'$.

\begin{leml} \label{lem:z'} 
The element $z_{2n+1}'$ defined in \eqref{z'} belongs to $\overline{N}_{2n+1}z_{2n+1} S_{2n+1}$, where $\overline{N}_{2n+1}$ is the unipotent radical of $\overline{B}_{2n+1}$. More precisely, there exists $u_0\in \overline{N}_{2n+1}$ such that 
$z_{2n+1}' = u_0 z_{2n+1} h_0^{-1}$, with
\[
h_0:=\begin{bmatrix} g_0 & {}^t e_n e_n  & {}^t e_n \\ 0& g_0 & 0\\ 0& e_n & 1\end{bmatrix},\qquad
g_0:= \left[\begin{smallmatrix} -2 & 1 \\ 1 & -2 & 1  \\ & \ddots & \ddots & \ddots  \\  & & 1 & -2 & 1 \\
 & & & 1 & -1 \end{smallmatrix}\right]_{n\times n}.
\]
\end{leml}

\begin{proof}
A direct calculation gives
\[
z'_{2n+1} h_0 z_{2n+1}^{-1}
=\begin{bmatrix} e_1' & & \\
g_0 & {}^t e_n e_1' & \\
0& w_n g_0 w_n & {}^te_n
\end{bmatrix},
\]
where $e_1':=(1,0,\dots,0)\in \bk^n$. This matrix clearly lies in $\overline{N}_{2n+1}$, as required.
\end{proof}

Let $g_0$ and $h_0$ be as in Lemma~\ref{lem:z'}. Note that $\det g_0 = (-1)^n$. Using Lemma~\ref{lem:z'} and Proposition~\ref{prop:Rodd} (2), the change of variables $h\mapsto \widehat h_0 \widehat h$ in \eqref{Lambda2}  yields
\[
\begin{aligned}
& \Lambda_{\rm JS}(1-s, \sigma'_{2n+1}\cdot \tilde f,  \CF_{\overline\psi}(\phi), \varphi_{2n+1}) \\
= &  \int_{S_{2n+1}} f(z_{2n+1} h_0^{-1} \widehat h)\left(R_{\varphi_{2n+1}}(h) \CF_{\overline\psi}(\phi)\right)(0)| h|_\bk^{\frac{1-s}{2}}\od\!h \\
 = &  \int_{S_{2n+1}} f(z_{2n+1} h ) \left(R_{\varphi_{2n+1}}(\widehat h_0) \CF_{\overline\psi}(R_{\varphi_{2n+1}^{-1}}(h)\phi)\right)(0)| h|_\bk^{\frac{s}{2}}\od\!h. 
\end{aligned}
\]

We now compute the action of $R_{\varphi_{2n+1}}(\widehat h_0)$. It is easy to verify that
\[
h_0 = \begin{bmatrix} 1_n & & {}^t e_n \\ & 1_n  \\ & & 1 \end{bmatrix} \begin{bmatrix} g_0  \\ & g_0 \\ & & 1\end{bmatrix} \begin{bmatrix} 1_n \\ & 1_n \\ & e_n & 1\end{bmatrix},
\]
so that
\[
\widehat h_0 =  \begin{bmatrix} 1_n   \\ & 1_n  \\ &  -e_n & 1 \end{bmatrix} \begin{bmatrix} {}^tg_0^{-1}  \\ & {}^tg_0^{-1} \\ & & 1\end{bmatrix} \begin{bmatrix} 1_n & & -{}^te_n\\ & 1_n \\ &  & 1\end{bmatrix}.
\]
Using Proposition~\ref{prop:Rodd} (1), we obtain that 
\[
\left(R_{\varphi_{2n+1}}(\widehat h_0)\phi\right) (0) = \eta(-1)^n  \psi( - e_n {}^tg_0^{-1} \, {}^te_n) \phi_1(- {}^te_n {}^t g_0^{-1})  = \eta(-1)^n \psi(n) \phi(v_n'),
\]
where  $\phi \in \CS(\bk^n)$ and $v_n':=(1,2,\dots,n)\in \bk^n$. Hence
\be \label{Lambda3}
\begin{aligned}
& \Lambda_{\rm JS}(1-s, \sigma'_{2n+1}\cdot  \tilde f,  \CF_{\overline\psi}(\phi), \varphi_{2n+1}) \\
 =\ &  \eta(-1)^n \psi(n)  \int_{S_{2n+1}} f(z_{2n+1} h ) \CF_{\overline\psi}(R_{\varphi_{2n+1}^{-1}}(h)\phi)(v_n')| h|_\bk^{\frac{s}{2}}\od\!h.
\end{aligned}
\ee

In the notation of \eqref{dag}, the diagonal torus $A_n$ of $G_n$ induces a diagonal torus $A_n^\ddag$ in $S_{2n+1}$. Define
\[
a':= u^{-1} a^\ddag u \quad\text{for }a\in A_n,
\]
where
\[
u: = \begin{bmatrix} u_n & \\ 0& u_n  \\0 & e_n & 1\end{bmatrix},\qquad
u_n := \left[\begin{smallmatrix}
1  \\ -1 & 1 \\ 
& \ddots & \ddots \\
& & -1 & 1
\end{smallmatrix}\right]_{n\times n}.
\]
Set $A_n' :=u^{-1} A_n^\ddag u=\{a'\mid a\in A_n\}$.

The following technical lemma will be used; its proof is a direct verification.

\begin{leml} \label{An'} 
For $a =\diag(a_1, a_2,\dots, a_n)\in A_n$, the element $z_{2n+1} a' z_{2n+1}^{-1}$ belongs to $\overline{B}_{2n+1}$ and has diagonal entries 
$a_1, a_2,\dots, a_n, 1, a_n,  \dots, a_2, a_1$.
\end{leml}

\begin{proof}
A direct computation yields
\[
\begin{aligned}
z_{2n+1}a' z_{2n+1}^{-1} 
&= z_{2n+1} u^{-1} a^\ddag u z_{2n+1}^{-1} \\
&= \begin{bmatrix} 1_n &  &  \\  & w_n & {}^tv_n \\  & 0 & 1\end{bmatrix} \begin{bmatrix} u_n^{-1} & \\ 0& u_n^{-1}  \\0 & -v_n & 1\end{bmatrix} \begin{bmatrix} a \\ & a \\ & & 1\end{bmatrix}
    \begin{bmatrix} u_n & \\ 0& u_n  \\0 & e_n & 1\end{bmatrix} \begin{bmatrix} 1_n &  &  \\  & w_n & -{}^tv_n \\  & 0 & 1\end{bmatrix}\\
&= \begin{bmatrix} u_n^{-1}a u_n &  &  \\
    0 & (w_n u_n^{-1} -{}^t v_n v_n)a u_n w_n + {}^tv_n e_1' &  \\
    0 & -v_n a u_n w_n + e_1' & a_1
\end{bmatrix},
\end{aligned}
\]
where $e_1'=(1,0,\ldots,0)\in \bk^n$. It is straightforward to check that the last matrix lies in $\overline{B}_{2n+1}$ with the stated diagonal entries.
\end{proof}

Let  $a= \diag(a_1,a_2,\ldots,a_n)\in A_n$. Again by Proposition~\ref{prop:Rodd} (2),  we have that 
\be \label{a1}
\CF_{\overline\psi}(R_{\varphi_{2n+1}^{-1}}(a' )\phi)  = | a|_\bk^{-1} R_{\varphi_{2n+1}}(\widehat{u^{-1} a^\ddag}) \CF_{\overline\psi}(R_{\varphi_{2n+1}^{-1}}(u)\phi).
\ee
Using Proposition~\ref{prop:Rodd} (1) and the factorization
\[
\widehat{u^{-1} a^\ddag} =\begin{bmatrix} 1_n & 0& {}^t e_n \\ & 1_n &0\\ & & 1\end{bmatrix}  \begin{bmatrix}  {}^t u_n a^{-1}  \\ & {}^t u_n a^{-1}  \\ & & 1\end{bmatrix},
\]
we obtain that for all $\phi_1\in \CS(\bk^n)$,
\be \label{a2}
\left(R_{\varphi_{2n+1}}(\widehat{u^{-1} a^\ddag})\phi_1\right)(v_n')   = \psi(-v_n' {}^t\!e_n) \eta( a)^{-1} \phi_1 (v_n'  
{}^t\!u_n  a^{-1}) = \psi(-n) \eta( a)^{-1} \phi_1(v_n a^{-1}).
\ee

Now write the integral in \eqref{Lambda3} as the iterated integral $\int_{A_n' \bsl S_{2n+1}} \int_{A_n'}$. Applying Lemma~\ref{An'}, \eqref{a1} and \eqref{a2}, we find that
\[
\begin{aligned}
& f(z_{2n+1}a' h ) \CF_{\overline\psi}(R_{\varphi_{2n+1}^{-1}}(a'h)\phi)(v_n')| a'h|_\bk^{\frac{s}{2}} \\
 = \,& \psi(-n)  f(z_{2n+1}  h) | h|_\bk^{\frac{s}{2}} \cdot\left(\prod^n_{i=1}(\xi_i\xi_{2n+2-i}\eta^{-1})(a_i)|a_i|_\bk^{s-1}\right) \cdot  \CF_{\overline\psi}(R_{\varphi_{2n+1}^{-1}}(uh)\phi)(a_1^{-1}, \dots, a_n^{-1}).
\end{aligned}
\]
After the change of variables $a\mapsto a^{-1}$ and applying Tate's thesis, we get
\[
\begin{aligned}
&  \int_{A_n'}  \left(\prod^n_{i=1} (\xi_i \xi_{2n+2-i}\eta^{-1})(a_i) |a_i|_\bk^{s-1}\right) \cdot  \CF_{\overline\psi}(R_{\varphi_{2n+1}^{-1}}(uh)\phi)(a_1^{-1}, \dots, a_n^{-1}) \od\! a'\\
= &   \prod^n_{i=1} \gamma(s, \xi_i \xi_{2n+2-i}\eta^{-1}, \overline\psi) \cdot \int_{A_n'} \left(\prod^n_{i=1}(\xi_i \xi_{2n+2-i}\eta^{-1})(a_i) |a_i|_\bk^s \right) \cdot \left(R_{\varphi_{2n+1}^{-1}}(uh)\phi\right)(a_1, \dots, a_n) \od\! a' \\
= &   \prod^n_{i=1} \gamma(s, \xi_i \xi_{2n+2-i}\eta^{-1},  \psi) \cdot \int_{A_n'} \left(\prod^n_{i=1}(\xi_i \xi_{2n+2-i}\eta^{-1})(a_i) |a_i|_\bk^s \right) \cdot \left(R_{\varphi_{2n+1}^{-1}}(uh)\phi\right)(-a_1, \dots, -a_n) \od\! a',
\end{aligned}
\]
where in the last step we made the substitution $a\mapsto -a$ and used the identity $\gamma(s,\omega,\overline\psi)=\omega(-1)\gamma(s,\omega,\psi)$ for $\omega\in \widehat{\bk^\times}$.

Observing that
\[
u^{-1}a =  \begin{bmatrix} 1_n \\ 0& 1_n \\ 0&  -e_n & 1\end{bmatrix}\begin{bmatrix} u_n^{-1}a \\ & u_n^{-1}a \\ & & 1\end{bmatrix}
\]
and $v_n u_n = e_n$, we have
\begin{eqnarray*}
\left(R_{\varphi_{2n+1}^{-1}}(a' h)\phi\right)(0)
&=&\eta^{-1}(a) \left(R_{\varphi_{2n+1}^{-1}}(uh)\phi\right) (-e_n u_n^{-1}a) \\
&=&\eta^{-1}( a)\left(R_{\varphi_{2n+1}^{-1}}(uh)\phi\right)(-a_1,  \dots, -a_n). 
\end{eqnarray*}
Consequently,
\[
\begin{aligned}
&\Lambda_{\rm JS}(1-s, \sigma'_{2n+1}\cdot \tilde f,   \CF_{\overline\psi}(\phi), \varphi_{2n+1}) \\ 
&\qquad\qquad=\eta(-1)^n \prod^n_{i=1}\gamma(s, \xi_i \xi_{2n+2-i}\eta^{-1},  \psi)  \\
&\qquad\qquad\qquad\qquad\cdot    \int_{A_n' \bsl S_{2n+1}}\int_{A_n'}
f(z_{2n+1}a' h) \left(R_{\varphi_{2n+1}^{-1}}(a' h)\phi\right)(0) | a'h|_\bk^{\frac{s}{2}}\od\! a' \od\! h  \\
&\qquad\qquad = \eta(-1)^n \cdot\left(  \prod^n_{i=1}\gamma(s, \xi_i \xi_{2n+2-i}\eta^{-1},  \psi)\right)  \cdot \Lambda_{\rm JS}(s, f, \phi, \varphi_{2n+1}^{-1}).
\end{aligned}
\]
This completes the proof of the functional equation \eqref{eq:FE'} in the odd case.

\section{Proof of Theorem \ref{thm:MF_m}} 
\label{sec:Ind}

In this section, we will simultaneously and inductively prove Theorem \ref{thm:MF_m} and the following theorem. 

\begin{thml} 
\label{thm2:FE_m}
For $(s,\xi)\in \Omega^m_\eta$ as in \eqref{minmax'}, $f\in I(\xi)$, and $\phi\in \mathcal{S}(\bk^n)$, the following functional equation holds: 
\begin{equation}\label{femjs}
\oZ_{\mathrm{JS}}(1-s, \sigma'_m\cdot W_{\tilde f}, \widehat{\phi}, \varphi_m)
=
\eta(-1)^{mn}\cdot \left(
\prod_{1\leq i<j \leq m} \gamma(s, \xi_i\xi_j\eta^{-1}, \psi)\right)
\cdot
\oZ_{\mathrm{JS}}(s, W_f, \phi, \varphi_m^{-1}),
\end{equation}
where
\[
\widehat{\phi} :=
\begin{cases}
\mathcal{F}_\psi(\phi), & \text{if $m$ is even},\\[4pt]
\mathcal{F}_{\overline{\psi}}(\phi), & \text{if $m$ is odd}.
\end{cases}
\]

\end{thml}

We label Theorem \ref{thm:MF_m} for $m$ by $({\rm MF}_{m})$, and label Theorem \ref{thm2:FE_m} for $m$ by $({\rm FE}_{m})$. 
 Consider the implication  
\be\label{imply3}
({\rm MF}_{m})+ ({\rm FE}_{m}) \Rightarrow ({\rm MF}_{m+1}).
\ee

\begin{proof}[Proof  of Theorems  \ref{thm:MF_m}  and \ref{thm2:FE_m} under \eqref{imply3}]
Note that $({\rm MF}_{0})$, $({\rm FE}_{0})$, $({\rm MF}_{1})$, and $({\rm FE}_{1})$  obviously  hold.  Since we have proved the functional equation \eqref{eq:FE'},  it is clear that 
\[
(\mathrm{MF}_m) \Rightarrow (\mathrm{FE}_m)\quad \textrm{for all $m\geq 1$}.
\] 
Under the assumption \eqref{imply3}, by induction this implies that both $(\mathrm{MF}_m)$ and $(\mathrm{FE}_m)$ holds for all $m\geq 1$. 
\end{proof}


The rest of this section is devoted to a proof of the implication \eqref{imply3}.

\subsection{Godement sections}\label{sec:GS} 
Assume that $({\rm MF}_{m})$ and  $({\rm FE}_{m})$ hold with $m\geq 1$, and write 
\[
\xi = (\xi_1,\xi_2,\dots, \xi_m) \in (\widehat{\bk^\times})^m\quad {\rm and}\quad \xi' = (\xi_1, \xi_2, \dots, \xi_m, \xi_{m+1})\in (\widehat{\bk^\times})^{m+1}. 
\]
Let $f' \in I(\xi')$ and $\phi \in \CS(\bk^{\lfloor (m+1)/2\rfloor})$ throughout the rest of this section. We need to show that $({\rm MF}_{m+1})$ holds for $I(\xi')$, that is,  
\be \label{MFm+1}
\Lambda_{\rm JS}(s,f',\phi,\varphi_{m+1}^{-1}) = \left(\prod_{1\leq i <j \leq m+1-i} \gamma(s, \xi_i \xi_j \eta^{-1},\psi)\right) \cdot 
\oZ_{\rm JS}(s, W_{f'}, \phi, \varphi_{m+1}^{-1})
\ee
where $(s, \xi') \in \widehat\Omega_\eta^{m+1}$.  Note that in this case the integrals on both sides of \eqref{MFm+1} converge absolutely (by Lemma~\ref{lem:CC}  and Corollary~\ref{cor:conv}), and that 
$(s, \xi')\in \widehat\Omega_\eta^{m+1}$ implies $(s, \xi) \in \Omega_\eta^m$. 

The basic idea for the proof of \eqref{MFm+1} is to apply the theory of Godement sections, for which we  recall some basic results from \cite{J09}. 
For $f\in I(\xi)$ and $\Phi \in \CS(\bk^{m\times (m+1)})$, set
\be\label{gs}
{\rm g}^+_{\Phi, f, \xi'}(h) :=    \xi_{m+1}( h) | h|_\bk^{\frac{m}{2}} 
\int_{G_m} \Phi([h_1 \mid 0]h) f(h_1^{-1}) \xi_{m+1}( h_1 ) |   h_1|_\bk^{\frac{m+1}{2}}\od\!h_1,
\ee
where $h\in G_{m+1}$, and $0$ denotes the zero vector in $\bk^{m\times 1}$. 
Let 
\[
\CY_m:=\Set{ Y\in \bk^{m\times (m+1)} | {\rm rank}\, Y = m }.
\]
As in \cite[Section 7.2]{J09}, there are natural left and right actions of $G_{m+1}$ and $G_m$, respectively, on $\CS(\bk^{m\times (m+1)})$, given by 
\[
(h\cdot\Phi\cdot h_1)(Y) := \Phi(h_1 Y h),\quad h\in G_{m+1}, \ h_1\in G_m, \  Y\in \bk^{m\times (m+1)},\ \Phi\in \CS(\bk^{m\times (m+1)}),
\]
and these actions clearly preserve $\CS(\CY_m)$. Here $\CS(\CY_m)$ denotes the space of Schwartz functions on $\CY_m$ (see \cite{AG08} for the notion of Schwartz functions on Nash manifolds).

The following  are consequences of  Propositions 7.1 and 7.2  in \cite{J09}. 

\begin{prpl}[Godement section] \label{prop:gs}

 \noindent {\rm (1)} If either $\Re(\xi_{m+1}) > \Re(\xi_i) - 1$ for all $i=1,2,\dots, m$, or $\Phi \in \CS(\CY_m)$, then \eqref{gs} converges absolutely  and defines an element of $I(\xi')$.  In this case, if $f' = {\rm g}^+_{\Phi, f, \xi'}\in I(\xi')$, then
\be \label{gsW}
\begin{aligned}
W_{f'}(h)   = \xi_{m+1}( h) | h|_\bk^{\frac{m}{2}} \int_{G_m}&\int_{\bk^{m}}\Phi(h_1[1_m \mid  {}^tz]h)\overline\psi(e_m {}^tz)\od\! z \\
& W_f(h_1^{-1}) \xi_{m+1}( h_1) | h_1|_\bk^{\frac{m+1}{2}}\od\! h_1,\quad h\in G_{m+1},
\end{aligned}
\ee
where the integral converges absolutely.

\noindent {\rm (2)} As a vector space, $I(\xi')$ is spanned by the functions 
${\rm g}^+_{\Phi, f, \xi'}$ with $f\in I(\xi)$ and $\Phi \in \CS(\CY_m)$. 
\end{prpl}

For a standard section $\{f_{\xi'}\}_{\xi'\in \CM}$ on a connected component $\CM$ of $(\widehat{\bk^\times})^{m+1}$ and $\phi\in \CS(\bk^{\lfloor (m+1)/2\rfloor})$, 
by Lemma~\ref{lem:CC}, 
the function 
\[
(s,\xi')\mapsto \Lambda_{\rm JS}(s,f_{\xi'},\phi,\varphi_{m+1}^{-1})
\]
is holomorphic on $\widehat\Omega_{\eta}^{m+1}\cap (\BC\times\CM)$;
and by Corollary~\ref{cor:conv}, the function 
\[
(s,\xi')\mapsto \oZ_{\rm JS}(s, W_{f_{\xi'}}, \phi, \varphi_{m+1}^{-1})
\]
is holomorphic on 
\[
\Set{(s,\xi')\in \BC\times \CM | \Re(s)>\Re(\eta) - 2\min_{i=1,\ldots,m+1}\Re(\xi_i)}.
\]
Then, by the uniqueness of holomorphic continuation, to prove that \eqref{MFm+1} holds for $(s,\xi')\in \widehat\Omega_\eta^{m+1}$, it suffices to prove that 
it holds for fixed $(s,\xi)\in \Omega_\eta^m$ and sufficiently large $\Re(\xi_{m+1})$. 

In the rest of this section, we assume that 
\[
(s,\xi)\in \Omega_\eta^m\quad\textrm{and}\quad \textrm{$\Re(\xi_{m+1})$ is  sufficiently large}. 
\]
In particular, 
the first condition in Proposition~\ref{prop:gs} (1) holds, namely
\[
\Re(\xi_{m+1}) > \Re(\xi_i) - 1\ \textrm{ for all $i=1,2,\dots, m$.}
\]
In view of Proposition~\ref{prop:gs} (2), for the purpose of proving \eqref{MFm+1} we assume without loss of generality that 
\be \label{f'gs}
f' = {\rm g}^+_{\Phi, f, \xi'},\quad \textrm{where}\ f \in I(\xi)\ {\rm and}\  \Phi\in \CS(\bk^{m\times (m+1)}).
\ee

To ease notation, for each  subgroup $\CG$ of $G_m$ put 
\[
\CG^+:=\set{ h^+ \mid h \in \CG} \subset G_{m+1},
\]
where for $h\in G_m$ we write 
$h^+ := \begin{bmatrix} h & 0 \\ 0 & 1 \end{bmatrix} \in G_{m+1}$.

We need to consider the even and odd cases for $m$ separately.

\subsection{The case $G_{2n}\to G_{2n+1}$}
Assume that $m=2n$.

\subsubsection{$\oZ_{\mathrm{JS}}$-side}
We start from $\oZ_{\mathrm{JS}}(s, W_{f'}, \phi, \varphi_{2n+1}^{-1})$. Define a subgroup of $S_{2n+1}$ by
\be \label{Sodd'}
S_{2n+1}' := \{\, h^+ \bar u_x \mid h\in S_{2n},\ x\in \bk^n \,\},
\ee
where
\be\label{ux}
\bar u_x := \begin{bmatrix}
1_n & 0 & 0 \\
0 & 1_n & 0 \\
0 & x & 1
\end{bmatrix},\qquad x\in \bk^n.
\ee
Define $\overline{S}_{2n+1}' := (\sigma_{2n+1}^{-1}N_{2n+1}\sigma_{2n+1}\cap S'_{2n+1}) \backslash S'_{2n+1}$.
Then the natural map
\[
S_{2n+1}'\hookrightarrow S_{2n+1}\twoheadrightarrow \overline{S}_{2n+1}\quad \text{(see \eqref{defsm})}
\]
induces a bijection $\overline{S}_{2n+1}' \xrightarrow{\sim} \overline{S}_{2n+1}$.

Note from \eqref{sigmam} that $\sigma_{2n+1} = \sigma_{2n}^+$, viewed as permutation matrices. The integral \eqref{JSodd} can also be written as
\be \label{JSodd'}
\begin{aligned}
\oZ_{\mathrm{JS}}(s, W, \phi,\varphi_{2n+1}^{-1})
&=
\int_{\overline{S}_{2n+1}'}
W(\sigma_{2n+1} h')
\left(R_{\varphi_{2n+1}^{-1}}(h')\phi\right)(0) |h'|_\bk^{\frac{s}{2}} \od\! h' \\
&=
\int_{\overline{S}_{2n}}
W_\phi( (\sigma_{2n} h)^+ ) \, \varphi_{2n}^{-1}(h) |h|_\bk^{\frac{s-1}{2}} \od\! h,
\end{aligned}
\ee
where $W_\phi\in \Ind_{N_{2n+1}}^{G_{2n+1}}\psi_{2n+1}$ denotes the action of $\phi$ on $W$ given by 
\[
W_\phi(h') := \int_{\bk^n} W(h' \bar u_x) \phi(x) \od\! x,\qquad h'\in G_{2n+1}.
\]
In the same vein, we will write $\Phi_\phi$ and $f'_\phi$ for similar actions of $\phi$ on $\Phi\in \CS(\bk^{2n\times (2n+1)})$ and $f'\in I(\xi')$, respectively.
By \eqref{gsW}, for $h\in S_{2n}$ we have
\[
\begin{aligned}
W_{f',\phi}( (\sigma_{2n} h)^+ )
&=
\xi_{2n+1}(\sigma_{2n} h) |h|_\bk^n
\int_{G_{2n}}
\int_{\bk^{2n}}
\Phi_\phi\bigl(h_1[1_{2n} \mid {}^t z](\sigma_{2n} h)^+\bigr)
\,\overline\psi(e_{2n}{}^t z) \od\! z \\
&\qquad W_f(h_1^{-1}) \xi_{2n+1}(h_1) |h_1|_\bk^{n+\frac{1}{2}} \od\! h_1.
\end{aligned}
\]

We note that
\[
h_1[1_{2n} \mid {}^t z](\sigma_{2n} h)^+
= [h_1 \sigma_{2n} h \mid h_1 {}^t z].
\]
After the change of variables $h_1\mapsto h_1(\sigma_{2n} h)^{-1}$ and $z\mapsto z\, {}^t(\sigma_{2n} h)$, we obtain
\[
\begin{aligned}
W_{f',\phi}( (\sigma_{2n} h)^+ )
&=
|h|_\bk^{\frac{1}{2}}
\int_{G_{2n}}
\int_{\bk^{2n}}
\Phi_{\phi, h_1}(z) \, \overline\psi(e_{2n} h \, {}^t z) \od\!z \\
&\qquad W_f(\sigma_{2n} h h_1^{-1}) \xi_{2n+1}(h_1) |h_1|_\bk^{n+\frac{1}{2}} \od\! h_1,
\end{aligned}
\]
where $\Phi_{\phi, h_1}\in \CS(\bk^{2n})$ is defined by
\be \label{Phi-res}
\Phi_{\phi, h_1}(z) := \Phi_\phi(h_1[1_{2n} \mid {}^t z]),\qquad z \in \bk^{2n}.
\ee

Recall the right action of $h = \begin{bmatrix} g & Xg   \\ 0& g \end{bmatrix}\in S_{2n}$ on $\bk^n$ given by \eqref{Sact}. Noting that \[
e_{2n}h = (0, e_n g) = (0, e_n\cdot h), 
\]
where $0$ is the zero vector in $\bk^n$, we obtain
\[
W_{f',\phi}((\sigma_{2n} h)^+)
=
|h|_\bk^{\frac{1}{2}}
\int_{G_{2n}}
\CF_{\overline\psi}(\Phi_{\phi, h_1})(0, e_n\cdot h)
W_f(\sigma_{2n} h h_1^{-1}) \xi_{2n+1}(h_1) |h_1|_\bk^{n+\frac{1}{2}} \od\!h_1.
\]
Plugging this into \eqref{JSodd'} for $W=W_{f'}$ yields an iterated integral
\[
\begin{aligned}
\oZ_{\mathrm{JS}}(s, W_{f'}, \phi, \varphi_{2n+1}^{-1})
&=
\int_{\overline{S}_{2n}}
\int_{G_{2n}}
\CF_{\overline\psi}(\Phi_{\phi, h_1})(0, e_n\cdot h)
W_f(\sigma_{2n} h h_1^{-1}) \xi_{2n+1}(h_1) |h_1|_\bk^{n+\frac{1}{2}} \od\!h_1 \\
&\qquad \varphi_{2n}^{-1}(h) |h|_\bk^{\frac{s}{2}} \od\!h.
\end{aligned}
\]
By Lemma~\ref{lem:switch} below and Fubini's theorem, we can switch the order of integration and obtain the recurrence relation
\be \label{rec}
\begin{aligned}
&\oZ_{\mathrm{JS}}(s, W_{f'}, \phi, \varphi_{2n+1}^{-1}) \\
&=
\int_{G_{2n}}
\int_{\overline{S}_{2n}}
W_f(\sigma_{2n} h h_1^{-1})
\CF_{\overline\psi}(\Phi_{\phi, h_1})(0, e_n\cdot h)
\varphi_{2n}^{-1}(h) |h|_\bk^{\frac{s}{2}} \od\!h
\,\xi_{2n+1}(h_1) |h_1|_\bk^{n+\frac{1}{2}} \od\!h_1 \\
&=
\int_{G_{2n}}
\oZ_{\mathrm{JS}}\bigl(s, W_{h_1^{-1}\cdot f}, \CF_{\overline\psi}(\Phi_{\phi, h_1})(0, \cdot), \varphi_{2n}^{-1}\bigr)
\,\xi_{2n+1}(h_1) |h_1|_\bk^{n+\frac{1}{2}} \od\!h_1.
\end{aligned}
\ee

\begin{leml} \label{lem:switch}
The double integral in \eqref{rec} converges absolutely when $(s,\xi)\in \Omega^{2n}_\eta$ is fixed and $\Re(\xi_{2n+1})$ is sufficiently large.
\end{leml}

\begin{proof}
Without loss of generality, assume that
\[
\Phi_\phi( X \mid {}^t z) = \Phi'(X) \phi'(z)\quad \textrm{for all $X\in \bk^{2n\times 2n}$ and $z\in \bk^{2n}$,}
\]
 where $\Phi'\in \CS(\bk^{2n\times 2n})$ and $\phi'\in \CS(\bk^{2n})$. Then from \eqref{Phi-res} we find that
\[
\CF_{\overline\psi}(\Phi_{\phi, h_1})(z)
=
\Phi'(h_1) \CF_{\overline\psi}(\phi')(z h_1^{-1}) |h_1|_\bk^{-1}.
\]
Thus, by Corollary~\ref{cor:conv}, it suffices to show that for every real number  $M>0$, the integral
\[
\int_{G_{2n}}
\| h_1\|_{\mathrm{HC}}^M \Phi'(h_1) \xi_{2n+1}(h_1) |h_1|_\bk^{n-\frac{1}{2}} \od\!h_1
\]
converges absolutely for $\Re(\xi_{2n+1})$ sufficiently large, where
\[
\|h_1\|_{\mathrm{HC}} := \|h_1\| + \|h_1^{-1}\|
\]
and $\|\cdot\|$ is the norm on $M_{2n}$ defined by
\[
\| [x_{i,j}]_{i,j=1,\dots,2n} \| := \sum_{i,j=1}^{2n} |x_{i,j}|_\bk
\]
(\cf \cite[Section 3.1]{J09} for the archimedean case). This is \cite[Lemma 3.3 (ii)]{J09}.
\end{proof}

For $z=(z_1,z_2)$ with $z_1,z_2\in \bk^n$,  write $\CF_{\psi'}^1$, $\CF_{\psi'}^2$ for the partial Fourier transforms on $\CS(\bk^{2n})$ with respect to the variables $z_1,z_2$ and a nontrivial unitary character $\psi'$ of $\bk$. Clearly, on $\CS(\bk^{2n})$ it holds that
\be \label{PF}
\CF_{\psi'} = \CF_{\psi'}^1 \circ \CF_{\psi'}^2 = \CF_{\psi'}^2 \circ \CF_{\psi'}^1.
\ee
In view of $(\mathrm{FE}_{2n})$ and \eqref{PF},  we have that 
\[
\begin{aligned}
& \gamma(s, I(\xi), \wedge^2\otimes\eta^{-1}, \psi)
\oZ_{\mathrm{JS}}\bigl(s, W_{h_1^{-1}\cdot f}, \CF_{\overline\psi}(\Phi_{\phi, h_1})(0, \cdot), \varphi_{2n}^{-1}\bigr) \\
&= \oZ_{\mathrm{JS}}\bigl(1-s, \sigma'_{2n}\cdot W_{{}^t h_1\cdot \tilde f}, \CF_{\overline\psi}^1(\Phi_{\phi, h_1})(0,\cdot), \varphi_{2n}\bigr).
\end{aligned}
\]
Applying $(\mathrm{MF}_{2n})$ for $\tilde{\xi} = (\xi_{2n}^{-1}, \dots, \xi_2^{-1}, \xi_1^{-1})$, and noting from Remark \ref{rmk:Omega} that
\[
(1-s, \widetilde{\xi})\in \Omega_{\eta^{-1}}^m,
\]
 we obtain
\[
\begin{aligned}
& \left(\prod_{1\leq i< j \leq 2n-i}
\gamma(1-s, \xi_{2n+1-i}^{-1} \xi_{2n+1-j}^{-1}\eta, \overline\psi)\right)\cdot 
\oZ_{\mathrm{JS}}\bigl(1-s, \sigma'_{2n}\cdot W_{{}^t h_1\cdot \tilde f}, \CF_{\overline\psi}^1(\Phi_{\phi, h_1})(0,\cdot), \varphi_{2n}\bigr) \\
&= \Lambda_{\mathrm{JS}}\bigl(1-s, \sigma'_{2n}{}^t h_1\cdot \tilde f, \CF_{\overline\psi}^1(\Phi_{\phi, h_1})(0,\cdot), \varphi_{2n}\bigr).
\end{aligned}
\]
Using $\gamma(s,\omega,\psi)\gamma(1-s,\omega^{-1},\overline\psi)=1$ for $\omega\in\widehat{\bk^\times}$, it is straightforward to check that
\[
\gamma(s, I(\xi), \wedge^2\otimes\eta^{-1}, \psi)
\prod_{1\leq i< j \leq 2n-i}
\gamma(1-s, \xi_{2n+1-i}^{-1} \xi_{2n+1-j}^{-1}\eta, \overline\psi)
=
\prod_{1\leq i<j\leq 2n+1-i}
\gamma(s, \xi_i\xi_j\eta^{-1}, \psi).
\]
From \eqref{rec} and the above calculations, we find that \eqref{MFm+1} for $m=2n$ is reduced to the recurrence relation
\be \label{rec'}
\begin{aligned}
\Lambda_{\mathrm{JS}}(s, f', \phi, \varphi_{2n+1}^{-1})
&=
\int_{G_{2n}}
\Lambda_{\mathrm{JS}}\bigl(1-s, \sigma'_{2n}{}^t h_1\cdot \tilde f, \CF_{\overline\psi}^1(\Phi_{\phi, h_1})(0,\cdot), \varphi_{2n}\bigr) \\
&\qquad \xi_{2n+1}(h_1) |h_1|_\bk^{n+\frac{1}{2}} \od\!h_1.
\end{aligned}
\ee

\subsubsection{$\Lambda_{\rm JS}$-side}

We now prove \eqref{rec'}. Recall that
\[
\Lambda_{\rm JS}(s, f', \phi, \varphi_{2n+1}^{-1})    
= \int_{S_{2n+1}} f'(z_{2n+1}h') \bigl(R_{\varphi_{2n+1}^{-1}}(h')\phi\bigr)(0) |h'|_\bk^{\frac{s}{2}}\od\! h',
\]
where $S_{2n+1} = \Set{ u_y  h^+ \bar u_x \mid h\in S_{2n},\ x,y\in \bk^n}$, with $\bar u_x$ as in \eqref{ux}, and
\be \label{uy}
u_y:= \begin{bmatrix} 1_n & & {}^ty \\ & 1_n & \\ & & 1\end{bmatrix},\qquad y\in \bk^n.
\ee

Let $u_y  h^+ \bar u_x\in S_{2n+1}$ be as above. Using Proposition~\ref{prop:Rodd} (1), we obtain
\[
\bigl(R_{\varphi_{2n+1}^{-1}}(u_y h^+ \bar u_x) \phi\bigr)(0) = \varphi_{2n}^{-1}(h) \phi(x).
\]
Consequently,
\be \label{Lambdaodd}
\Lambda_{\rm JS}(s, f', \phi, \varphi_{2n+1}^{-1})    
= \int_{S_{2n}} \int_{\bk^{n}} f'_\phi\bigl(z_{2n+1} u_y h^+ \bigr) \od\! y \,
\varphi_{2n}^{-1}(h) |h|_\bk^{\frac{s-1}{2}} \od\! h.
\ee

By \eqref{gs}, we have
\[
\begin{aligned}
f'_\phi(z_{2n+1} u_y h^+ ) 
&= \xi_{2n+1}(z_{2n+1}h^+) |h|_\bk^n \\
&\quad \cdot \int_{G_{2n}} \Phi_\phi\bigl((h_1 \mid 0)z_{2n+1} u_y h^+ \bigr) f(h_1^{-1}) \xi_{2n+1}(h_1) |h_1|_\bk^{n+\frac{1}{2}}\od\! h_1.
\end{aligned}
\]
A direct computation gives
\[
[h_1 \mid 0] z_{2n+1} u_y h^+ = [h_1 \mid 0]
\begin{bmatrix} z_{2n} h  & {}^t(y, v_n) \\ & 1\end{bmatrix}
= h_1 [z_{2n} h \mid {}^t(y, v_n)].
\]
Making the change of variables $h_1 \mapsto h_1 (z_{2n}h)^{-1}$, and using the identities
$\det z_{2n+1}= \det z_{2n}$ and
$(y, v_n){}^t(z_{2n}h)^{-1} = (y, v_n)z_{2n} {}^t\!h^{-1} = (y, v_n) {}^t\!h^{-1}$,
we obtain
\[
f'_\phi(z_{2n+1} u_y h^+ ) 
= |h|_\bk^{-\frac{1}{2}}
\int_{G_{2n}} \Phi_{\phi, h_1}\bigl((y, v_n){}^t\!h^{-1}\bigr)
f(z_{2n} h h_1^{-1}) \xi_{2n+1}(h_1) |h_1|_\bk^{n+\frac{1}{2}}\od\! h_1.
\]

The order of integration over $h_1\in G_{2n}$ and $y\in \bk^n$ in \eqref{Lambdaodd} may be interchanged. Then, for $h\in S_{2n}$ as in \eqref{Sact}, an affine change of variable in $y$ yields
\[
\int_{\bk^n}\Phi_{\phi, h_1}\bigl((y, v_n){}^t\!h^{-1}\bigr) \od\! y
= |g|_\bk \int_{\bk^n}\Phi_{\phi, h_1}(y, v_n {}^t\!g^{-1}) \od\! y
= |h|_\bk^{\frac{1}{2}} \,\CF_{\overline\psi}^1(\Phi_{\phi,h_1})(0, v_n\cdot \widehat h).
\]
Therefore,
\[
\begin{aligned}
\Lambda_{\rm JS}(s, f', \phi, \varphi_{2n+1}^{-1})    
&= \int_{S_{2n}} \int_{G_{2n}} f(z_{2n}h h_1^{-1})
   \CF^1_{\overline\psi}(\Phi_{\phi,h_1})(0, v_n\cdot\widehat h)
   \xi_{2n+1}(h_1) |h_1|_\bk^{n+\frac{1}{2}}\od\! h_1 \\
&\qquad \cdot \varphi_{2n}^{-1}(h) |h|_\bk^{\frac{s-1}{2}} \od\! h.
\end{aligned}
\]

Assuming absolute convergence, we switch the order of integration to get
\be\label{doub'}
\begin{aligned}
\Lambda_{\rm JS}(s, f', \phi, \varphi_{2n+1}^{-1})    
&= \int_{G_{2n}} \int_{S_{2n}} f(z_{2n}h h_1^{-1})
   \CF^1_{\overline\psi}(\Phi_{\phi,h_1})(0, v_n\cdot\widehat h)
   \varphi_{2n}^{-1}(h) |h|_\bk^{\frac{s-1}{2}} \od\! h \\
&\qquad \cdot \xi_{2n+1}(h_1) |h_1|_\bk^{n+\frac{1}{2}}\od\! h_1.
\end{aligned}
\ee

On the other hand,
\[
\begin{aligned}
&\Lambda_{\rm JS}(1-s, \sigma'_{2n}{}^th_1\cdot\tilde f, \CF_{\overline\psi}^1(\Phi_{\phi, h_1})(0,\cdot), \varphi_{2n}) \\
&= \int_{S_{2n}} f(w_{2n} z_{2n} {}^th^{-1} \sigma'_{2n} h_1^{-1})
   \CF_{\overline\psi}^1(\Phi_{\phi, h_1})(0, v_n\cdot h)
   \varphi_{2n}(h) |h|_\bk^{\frac{1-s}{2}}\od\! h \\
&= \int_{S_{2n}} f(z_{2n}\widehat h h_1^{-1})
   \CF_{\overline\psi}^1(\Phi_{\phi, h_1})(0, v_n\cdot h)
   \varphi_{2n}(h) |h|_\bk^{\frac{1-s}{2}}\od\! h \\
&= \int_{S_{2n}} f(z_{2n} h h_1^{-1})
   \CF_{\overline\psi}^1(\Phi_{\phi, h_1})(0, v_n\cdot \widehat h)
   \varphi_{2n}^{-1}(h) |h|_\bk^{\frac{s-1}{2}}\od\! h.
\end{aligned}
\]

The same arguments as in the proof of Lemma~\ref{lem:switch}, together with $({\rm MF}_{2n})$, show that \eqref{doub'} is absolutely convergent. This proves \eqref{rec'}, and hence completes the proof of \eqref{MFm+1} for $m=2n$.

\subsection{The case $G_{2n+1}\to G_{2n+2}$} Assume that $m=2n+1$.  

\subsubsection{$\oZ_{\rm JS}$-side}

We first record some group-theoretic preliminaries. From \eqref{sigmam} one readily verifies that
\be\label{varsig}
\sigma_{2n+2} = \sigma_{2n+1}^+\varsigma_n^+,\quad \text{where}\quad \varsigma_n:=
\begin{bmatrix}
1_{n}\\
& 0& 1_{n} \\
& 1 &0
\end{bmatrix}\in G_{2n+1}.
\ee
Let $S_{2n+1}'$ be the subgroup of $S_{2n+1}$ defined in \eqref{Sodd'}, and put
\[
T_n:=\varsigma_n^{-1} S_{2n+1}' \varsigma_n = \Set{ \begin{bmatrix} g & 0 & Xg  \\ & 1 & x \\ & & g \end{bmatrix}  | 
 \begin{array}{l} g\in G_{n}, X\in M_{n} \\ x \in \bk^{1\times n} \end{array}}.
\]
Then $T_n^+\subset S_{2n+2}$. The map 
\[
S_{2n+1}' \to S_{2n+2}, \quad h\mapsto (\varsigma_n^{-1} h \varsigma_n)^+
\]
induces an embedding 
$\overline{S}_{2n+1}'\hookrightarrow \overline{S}_{2n+2}$; we denote its image by $\overline{T}_n^+ \subset \overline{S}_{2n+2}$.

Define a subgroup $R_n$ of $G_{n+1}$ by
\[
R_n :=\Set{  \begin{bmatrix} 1_{n} \\  v & a \end{bmatrix} | a\in\bk^\times,  v\in \bk^{n} },
\]
so that $\overline{P}_{n,1}:=G_{n}^+ R_n$ is the block lower triangular maximal parabolic subgroup of $G_{n+1}$ of type $(n,1)$. 
It is easy to see that $\overline{P}_{n,1}^\dag$ (as in \eqref{dag}) normalizes the unipotent radical of $T_n^+$; consequently, 
$T_n^+ R_n ^\dag$ is a subgroup of $S_{2n+2}$. Moreover, the multiplication map 
$T_n^+\times R_n ^\dag \to T_n^+ R_n ^\dag$ is bijective, and the multiplication map 
$\overline{T}_n ^+\times R_n ^\dag \to \overline{S}_{2n+2}$ is an embedding with open dense image. 
It follows that the integral \eqref{JSeven} can be rewritten as
\be \label{JSeven'}
\begin{aligned}
& \oZ_{\rm JS}(s, W, \phi, \varphi_{2n+2}^{-1}) \\
 = &  \int_{ R_n } \int_{\overline{T}_n^+  } W(\sigma_{2n+2} h r^\dag)  \phi(e_{n+1}\cdot h r^\dag)  \varphi_{2n+2}^{-1}(h r^\dag) |h|_\bk^{\frac{s-1}{2}} |r|_\bk^{s}\od\! h \od\! r \\
 = &  \int_{R_n }\int_{\overline{S}_{2n+1}'}W ((\sigma_{2n+1} h \varsigma_n)^+r^\dag) \phi(e_{n+1} r) \varphi_{2n+1}'^{-1}(h)\eta^{-1}(r) |h|_\bk^{\frac{s-1}{2}} |r|_\bk^{s}\od\! h \od\! r,
\end{aligned}
\ee
where $\varphi_{2n+1}'$ is the character of $S_{2n+1}'$ given by
\be \label{S'}
h = \begin{bmatrix} g & Xg &0\\ & g&0 \\ & x & 1\end{bmatrix} \mapsto \eta(g)\psi({\rm tr}\, X),\qquad g\in G_{n},\ X\in M_{n},\ x\in \bk^{n}. 
\ee

Now take $f' ={\rm g}^+_{\Phi, f, \xi'}$ as in \eqref{f'gs}. For $h\in S_{2n+1}'$ and $r\in R_n$, using \eqref{gsW} we obtain
\[
\begin{aligned}
W_{f'}((\sigma_{2n+1} h \varsigma_n)^+r^\dag) = \ &  \xi_{2n+2}((\sigma_{2n+1} h \varsigma_n)^+ r^\dag) | h^+ r^\dag|_\bk^{n+\frac{1}{2}}\\
&  \begin{aligned}\cdot \int_{G_{2n+1}} &  \int_{\bk^{2n+1}}\Phi(h_1[1_{2n+1} \mid {}^t z] (\sigma_{2n+1} h \varsigma_n)^+r^\dag)\overline\psi(e_{2n+1}{}^t z)\od\! z \\
 & W_f(h_1^{-1})\xi_{2n}(h_1)|h_1|_\bk^n \od\! h_1.
 \end{aligned}
\end{aligned}
\]
Observe that
\[
h_1[1_{2n+1} \mid {}^t z] (\sigma_{2n+1} h \varsigma_n)^+ = [h_1 \sigma_{2n+1} h  \varsigma_n \mid h_1 {}^tz].
\]
Making the changes of variables $h_1\mapsto h_1(\sigma_{2n+1} h\varsigma_n)^{-1}$ and $z\mapsto z\, {}^t(\sigma_{2n+1} h\varsigma_n)$, and writing $h\in S_{2n+1}'$ as in \eqref{S'}, a direct calculation gives
\[
e_{2n+1} \sigma_{2n+1}  h\varsigma_n = (e_{n+1}, x) \in \bk^{2n+1}.
\]
Consequently,
\[
\begin{aligned}
W_{f'}((\sigma_{2n+1} h \varsigma_n)^+r^\dag)  = \ & \xi_{2n+2}^2(r) |r|_\bk^{2n+1} |h|_\bk^{\frac{1}{2}}\\
& \int_{G_{2n+1}}  \CF_{\overline\psi}(\Phi_{r, h_1})(e_{n+1}, x) W_f(\sigma_{2n+1} h \varsigma_n h_1^{-1})
\xi_{2n+2}(h_1) |h_1|_\bk^{n+1} \od\! h_1,
\end{aligned}
\]
where $\Phi_{r, h_1}\in \CS(\bk^{2n+1})$ is defined by
\[
\Phi_{r, h_1}(z) := \Phi(h_1[1_{2n+1} \mid {}^t z]r^\dag),\qquad z\in \bk^{2n+1}.
\]

Substituting the above expression for $W_{f'}((\sigma_{2n+1} h \varsigma_n)^+r^\dag)$ into \eqref{JSeven'} yields
\[
\begin{aligned}
 \oZ_{\rm JS}(s, W_{f'}, \phi, \varphi_{2n+2}^{-1})  = \int_{R_n } \int_{\overline{S}_{2n+1}'}& \int_{G_{2n+1}}
W_{\varsigma_n h_1^{-1}\cdot f}(\sigma_{2n+1} h)  \CF_{\overline\psi}(\Phi_{r, h_1})(e_{n+1}, x) \xi_{2n+2}(h_1) |h_1|_\bk^{n+1} \od\! h_1 \\
& \varphi_{2n+1}'^{-1}(h) |h|_\bk^{\frac{s}{2}} \od\! h   \,  \phi(e_{n+1}r) \xi_{2n+2}^2\eta^{-1}(r)    |r|_\bk^{s+2n+1} \od\! r.
 \end{aligned}
\]
As in Lemma~\ref{lem:switch}, we may interchange the order of integration to obtain the recurrence relation
\be \label{recodd}
\begin{aligned}
&\oZ_{\rm JS}(s, W_{f'}, \phi, \varphi_{2n+2}^{-1}) \\
= &   \int_{R_n }   \int_{G_{2n+1}}   \int_{\overline{S}_{2n+1}'}
W_{\varsigma_n h_1^{-1}\cdot f}(\sigma_{2n+1} h)  \CF_{\overline\psi}(\Phi_{r, h_1})(e_{n+1}, x) \varphi_{2n+1}'^{-1}(h) |h|_\bk^{\frac{s}{2}} \od\! h \\
& \qquad \qquad\qquad\xi_{2n+2}(h_1) |h_1|_\bk^{n+1} \od\! h_1 \, \phi(e_{n+1} r) \xi_{2n+2}^2 \eta^{-1}(r)   |r|_\bk^{s+2n+1} \od\! r\\
= & \int_{R_n }\int_{G_{2n+1}}  \oZ_{\rm JS}(s, W_{\varsigma_n h_1^{-1}\cdot f}, \CF_{\overline\psi}(\Phi_{r, h_1})(e_{n+1}, \cdot), \varphi_{2n+1}^{-1}) \\
 &\qquad  \qquad\qquad\xi_{2n+2}(h_1) |h_1|_\bk^{n+1} \od\! h_1\, \phi(e_{n+1} r) \xi_{2n+2}^2  \eta^{-1}(r)   |r|_\bk^{s+2n+1} \od\! r,
 \end{aligned}
\ee
where we have used \eqref{JSodd'} and \eqref{S'}.

As in the even case, for $z=(z_1, z_2)$ with $z_1\in \bk^{n+1}$ and $z_2\in \bk^{n}$, denote by $\CF_{\psi'}^1$ and $\CF_{\psi'}^2$ the partial Fourier transforms on $\CS(\bk^{2n+1})$ with respect to the variables $z_1$ and $z_2$, respectively, for a nontrivial unitary character $\psi'$ of $\bk$. 
Applying (${\rm FE}_{2n+1}$) for $\xi$ and (${\rm MF}_{2n+1}$) for $\tilde\xi$, 
we see that \eqref{MFm+1} for $m=2n+1$ reduces to the recurrence relation
\be \label{recodd'}
\begin{aligned}
& \Lambda_{\rm JS}(s, f', \phi, \varphi_{2n+2}^{-1}) \\
=\ & \eta(-1)^{n}  \int_{R_n }   \int_{G_{2n+1}}   \Lambda_{\rm JS}(1-s, \sigma'_{2n+1} \varsigma_n {}^t h_1\cdot \tilde f, \CF_{\overline\psi}^1(\Phi_{r, h_1}^-)(e_{n+1},\cdot ), \varphi_{2n+1}) \\
 & \qquad\qquad\qquad \xi_{2n+2}(h_1) |h_1|_\bk^{n+1} \od\! h_1\, \phi(e_{n+1} r) \xi_{2n+2}^2 \eta^{-1}(r)   |r|_\bk^{s+2n+1} \od\! r,
\end{aligned}
\ee
with $\Phi_{r, h_1}^-(z_1, z_2):=\Phi_{r, h_1}(z_1, -z_2)$.

\subsubsection{$\Lambda_{\rm JS}$-side}

We now prove the recurrence relation \eqref{recodd'}. Recall the base point
$x_{2n+2} = (\overline{B}_{2n+2}z_{2n+2}, v_{n+1})$ of the open $S_{2n+2}$-orbit in $\CX_{2n+2}$ given by \eqref{basept}.
For computational convenience, we introduce a new base point. Let
\[
z_{2n+1}' = \begin{bmatrix} -v_{n} & 0& 1 \\ 1_{n} & 0&0\\ 0&  w_{n}&0 \end{bmatrix} \in G_{2n+1}
\]
be as in \eqref{z'}, and set
$
g_n = \begin{bmatrix} -v_{n} & 1  \\ 1_{n} & 0 \end{bmatrix}\in G_{n+1}.
$
A direct check gives
\be \label{zodd'}
(z_{2n+2}g_n^\dag, v_{n+1}\cdot g_n^\dag) = (z_{2n+2}', e_{n+1}),
\quad \text{where}\quad 
z_{2n+2}' = \begin{bmatrix} -v_{n} & 1 \\ 1_{n} &0 \\ & & w_{n}&0 \\ & & -v_{n} & 1\end{bmatrix},
\ee
and clearly
$[1_{2n+1} \mid 0 ] z_{2n+2}' = [z_{2n+1}' \varsigma_n \mid 0]$.
Since $\det g_n = (-1)^n$, we obtain
\be \label{newlambda}
\begin{aligned}
\Lambda(s, f',\phi, \varphi_{2n+2}^{-1})
&= \int_{S_{2n+2}} f'(z_{2n+2}h') \phi(v_{n+1}\cdot h') \varphi_{2n+2}^{-1}(h') |h'|_\bk^{\frac{s}{2}}\od\!h'\\
&= \eta(-1)^{n}\int_{S_{2n+2}} f'(z_{2n+2}' h') \phi(e_{n+1}\cdot h') \varphi_{2n+2}^{-1}(h') |h'|_\bk^{\frac{s}{2}}\od\!h'.
\end{aligned}
\ee

We now decompose the integration domain $S_{2n+2}$. Recall the subgroup $T_n^+ R_n ^\dag$ of $S_{2n+2}$, and let $U_{n+1}$ denote the unipotent radical of the mirabolic subgroup $P_{n+1}$ of $G_{n+1}$:
\[
U_{n+1} = \Set{ u_y' = \begin{bmatrix} 1_{n}  & {}^ty \\ & 1\end{bmatrix} | y \in \bk^{n}}.
\]
Also set
\[
V_{n+1} = \Set{ v_z = \begin{bmatrix} 1_{n+1} & 0&  {}^t z \\   & 1_{n} &0 \\ & & 1 \end{bmatrix} | z\in \bk^{n+1}}.
\]
It is straightforward to verify that the multiplication map
\be \label{4gp}
U_{n+1}^\dag \times T_n^+ \times V_{n+1} \times R_n ^\dag \to S_{2n+2}
\ee
is an embedding with open dense image. We shall integrate over this image.

Let $u_y'\in U_{n+1}$ and $v_z\in V_{n+1}$ be as above. Recall that $T_n=\varsigma_n^{-1}S_{2n+1}'\varsigma_n$. For $h\in S'_{2n+1}$, $r\in R_n$, write
\be\label{h'decom}
h' := u_y'^\dag\, (\varsigma_n^{-1} h \varsigma_n)^+\, v_z \, r^\dag \in S_{2n+2}
\ee
as in \eqref{4gp}. 
Since $U_{n+1}^\dag T_n^+ V_{n+1}\subset P_{2n+2}$, we have
\be \label{chardecom}
e_{n+1}\cdot h' = e_{n+1}r,\qquad
\varphi_{2n+2}(h') = \varphi_{2n+1}'(h)\, \psi(e_{n+1} {}^t z)\, \eta(r),
\ee
where $\varphi_{2n+1}'$ is the character of $S_{2n+1}'$ defined in \eqref{S'}.

Using \eqref{gs}, we get
\[
f'(z_{2n+2}'h') = \xi_{2n+2}(z_{2n+2}' h') |h'|_\bk^{n+\frac{1}{2}}
\int_{G_{2n+1}} \Phi(h_1 [1_{2n+1} \mid 0]z_{2n+2}' h') f(h_1^{-1})
\xi_{2n+2}(h_1) |h_1|_\bk^{n+1} \od\!h_1.
\]
A direct computation shows that for $h'$ as in \eqref{h'decom},
\[
[1_{2n+1} \mid 0]z_{2n+2}' h'
= [z_{2n+1}' \varsigma_n \mid 0] u_y'^\dag (\varsigma_n^{-1} h \varsigma_n)^+ v_z r^\dag
= [z_{2n+1}'u_y h \varsigma_n \mid {}^t z_{h'}] r^\dag,
\]
where $u_y$ is as in \eqref{uy} and
\[
{}^t z_{h'} = z_{2n+1}' u_y h \varsigma_n \begin{bmatrix} {}^t z \\0\end{bmatrix}
+ \begin{bmatrix} 0 \\ w_{n} {}^t y\end{bmatrix} \in \bk^{(2n+1)\times 1}.
\]
We change variables $h_1 \mapsto h_1(z_{2n+1}' u_y h \varsigma_n)^{-1}$ in the above integral. At this point, a lengthy but straightforward calculation is required. Write
\[
h = \begin{bmatrix} g & Xg&0 \\0& g &0\\ 0& x & 1\end{bmatrix} \in S_{2n+1}'
\]
as in \eqref{S'}. Then one obtains
\[
(z_{2n+1}' u_y h \varsigma_n )^{-1}[1_{2n+1} \mid 0]z_{2n+2}' h'
= [1_{2n+1} \mid {}^t z_{h'}']r^\dag,
\]
where
\[
{}^t z_{h'}' = \begin{bmatrix} {}^t z \\ 0\end{bmatrix}
- \begin{bmatrix} g^{-1} X \,{}^t y \\  x g^{-1}\, {}^t y \\  -g^{-1} \, {}^t y\end{bmatrix}.
\]
We then perform the change of variables
$z \mapsto z + (y\, {}^t X \, {}^t g^{-1},\, y \, {}^t g^{-1}\, {}^t x)$ in \eqref{newlambda}.
Recalling the right action of $S_{2n+1}$ on $\bk^{n}$ from \eqref{Sact'} and the involution in \eqref{inv}, one verifies that
$- y\, {}^t g^{-1} = 0\cdot \widehat{u_y h}$.

Using \eqref{chardecom} and the identity $\det z_{2n+2}' = \det (z_{2n+1}' \varsigma_n )$, after the above substitutions we obtain
\[
\begin{aligned}
 \Lambda(s, f',\phi, \varphi_{2n+2}^{-1})
 =\ & \eta(-1)^{n} \int_{R_n} \int_{S_{2n+1}'} \int_{\bk^{n}}
 \int_{\bk^{n+1}} \overline\psi(e_{n+1} {}^t z) \\
& \int_{G_{2n+1}} \Phi_{r, h_1}^-(z, 0\cdot \widehat{u_y h})
   f(z_{2n+1}' u_y h \varsigma_n h_1^{-1}) \xi_{2n+2}(h_1) |h_1|_\bk^{n+1}\od\!h_1 \od\!z \\
& \qquad \psi((0\cdot\widehat{u_y h}) {}^t x) \varphi_{2n+1}'^{-1}(h) |h|_\bk^{\frac{s-1}{2}} \od\!y \od\!h
\, \phi(e_{n+1} r) \xi_{2n+2}^2 \eta^{-1}(r) |r|_\bk^{s+2n+1}\od\!r.
\end{aligned}
\]
Assuming absolute convergence (which will be justified shortly), we interchange the order of integration to get
\be \label{doubodd'}
\begin{aligned}
\Lambda(s, f',\phi, \varphi_{2n+2}^{-1})
= &\, \eta(-1)^{n} \int_{R_n} \int_{G_{2n+1}}
\int_{S_{2n+1}'}\int_{\bk^{n}} f(z_{2n+1}' u_y h \varsigma_n h_1^{-1}) \\
&\qquad \CF^1_{\overline\psi}(\Phi^-_{r, h_1})(e_{n+1}, 0\cdot\widehat{u_yh})
\, \psi((0\cdot\widehat{u_y h})\, {}^t x) \varphi_{2n+1}'^{-1}(h) |h|_\bk^{\frac{s-1}{2}} \od\!y \od\!h \\
&\qquad \qquad \xi_{2n+2}(h_1) |h_1|_\bk^{n+1} \od\!h_1
\, \phi(e_{n+1} r) \xi_{2n+2}^2 \eta^{-1}(r) |r|_\bk^{s+2n+1} \od\!r.
\end{aligned}
\ee

On the other hand, since $S_{2n+1} = \set{ u_y h \mid h\in S_{2n+1}',\ y\in \bk^{n} }$, using \eqref{Lambda2} and the identity ${}^t\varsigma_n^{-1}=\varsigma_n$, we obtain for any $\phi_1\in \CS(\bk^{n+1})$:
\[
\begin{aligned}
& \Lambda_{\rm JS}(1-s, \sigma'_{2n+1} \varsigma_n {}^t h_1\cdot \tilde f, \phi_1, \varphi_{2n+1}) \\
= & \int_{S_{2n+1}'}\int_{\bk^{n}} f(z_{2n+1}' \widehat{u_y h }\varsigma_n h_1^{-1})
\bigl(R_{\varphi_{2n+1}}(u_y h)\phi_1\bigr)(0) |h|_\bk^{\frac{1-s}{2}}\od\!y \od\!h \\
= & \int_{S_{2n+1}'}\int_{\bk^{n}} f(z_{2n+1}' u_y h \varsigma_n h_1^{-1})
\bigl(R_{\varphi_{2n+1}}(\widehat{u_y h})\phi_1\bigr)(0) |h|_\bk^{\frac{s-1}{2}}\od\!y \od\!h.
\end{aligned}
\]
For $u_y h\in S_{2n+1}$ as above, Proposition~\ref{prop:Rodd} (1) gives
\[
R_{\varphi_{2n+1}}(\widehat{u_y h})\phi_1(0)
= \phi_1(0\cdot\widehat{u_yh})\, \psi((0\cdot\widehat{u_y h})\, {}^t x) \varphi_{2n+1}'^{-1}(h).
\]
Now set $\phi_1 := \CF^1_{\overline\psi}(\Phi^-_{r, h_1})(e_{n+1}, \cdot)$. The same arguments as in the proof of Lemma~\ref{lem:switch}, together with $({\rm MF}_{2n+1})$, show that \eqref{doubodd'} is absolutely convergent. This proves \eqref{recodd'}, and hence completes the proof of \eqref{MFm+1} for $m=2n+1$.
   
\section{Proof of Theorem \ref{thm:FE_m} and Theorem~\ref{thm:FE'_m} (2)} \label{sec:red}

Theorem~\ref{thm:FE_m} is trivial for $m=0$, so we assume $m\geq 1$. 
Fix an inducing datum $(\overline P,\tau)$ as in \eqref{nt}, and set
\[
\mathcal{M}_r := \left\{ \lambda\in \mathbb{C}^r : \max \Re(\lambda) < \min \Re(\lambda) + 1/2 \right\}.
\]
For every nonempty open subset $\mathcal{N}\subset \mathcal{M}_r$, define
\[
\Omega_{\mathcal{N}} := \left\{ (s,\lambda)\in \mathbb{C}\times \mathbb{C}^r : 
\lambda\in \mathcal{N},\ 
s \in \mathcal{H}_{\left(\Re(\eta)-2\min\Re(\lambda),\ \Re(\eta)+1-2\max\Re(\lambda)\right)}
\right\},
\]
which is a nonempty open subset of $\mathbb{C}\times \mathbb{C}^r$.

\begin{leml} \label{lem:red}
Assume that there exists a nonempty open subset $\CN\subset \CM_r$ such that \eqref{eq:FE_m} holds for all $(s,\lambda)\in \Omega_\CN$. 
Then Theorem~\ref{thm:FE_m} holds for the pair $(\overline P,\tau)$.
\end{leml}

\begin{proof}
Recall that Part~(1) of Theorem~\ref{thm:FE_m} follows from Corollary~\ref{cor:conv}.

Following the argument in \cite[Section~3.10]{BP21}, there exist constants $u\in \BC^\times$, $C\in \mathbb{R}^\times_+$, and a linear functional $\ell$ on $\BC^r$ such that
\[
\eta(-1)^{mn} \varepsilon(s, \pi_\lambda, \wedge^2\otimes \eta^{-1}, \psi) = u\cdot C^{\ell(\lambda)+s-\frac{1}{2}}\qquad \textrm{for all }\, \lambda\in \BC^r,\ s\in \BC.
\]
Choose a square root $v$ of $u$ and set
\[
\epsilon_{1/2}(s, \pi_\lambda, \wedge^2\otimes\eta^{-1}, \psi):= v\cdot \sqrt{C}^{\ell(\lambda)+s-\frac{1}{2}},
\]
so that
\[
\eta(-1)^{mn} \varepsilon(s, \pi_\lambda, \wedge^2\otimes \eta^{-1}, \psi)
= \epsilon_{1/2}(s, \pi_\lambda, \wedge^2\otimes\eta^{-1}, \psi)^2.
\]

For $f\in \pi_\tau$ and $\phi \in \CS(\bk^n)$, define
\[
\begin{aligned} 
\oZ_+(s,\lambda)  & := \epsilon_{1/2}(s, \pi_\lambda,\wedge^2\otimes\eta^{-1},  \psi)\,
   \oZ_{\rm JS}^\circ\left(\frac{1}{2}+s, W_{f_\lambda}, \phi,\varphi_m^{-1}\right),\\
\oZ_-(s,\lambda) & :=  \epsilon_{1/2}(s, \pi_\lambda, \wedge^2\otimes\eta^{-1}, \psi)^{-1}\,
   \oZ_{\rm JS}^\circ\left(\frac{1}{2}+s, \sigma'_m\cdot \widetilde{W}_{f_\lambda}, \widehat\phi,\varphi_m\right)
\end{aligned}
\]
whenever $(s,\lambda)$ satisfies the convergence conditions in \eqref{ss'}. By the assumption of the lemma, \eqref{eq:FE_m} holds on $\Omega_\CN$. Put
\[
\Omega_\CN':=\set{(s-\tfrac12,\lambda) \mid (s,\lambda)\in \Omega_\CN},
\]
and choose a subset $\CH_{(a,b)}\times U\subset \Omega_\CN'$, where $a<b$ are real numbers and $U\subset \BC^r$ is a nonempty relatively compact connected open subset. Then
\be \label{FE}
\oZ_+(s, \lambda) =  \oZ_-(-s, \lambda)
\ee
for every $(s,\lambda)\in \CH_{(a,b)}\times U$.

By Lemma~\ref{lem:Lebesgue} and the proof of Proposition~\ref{prop:conv0}, the functions $\oZ_+(s,\lambda)$ and $\oZ_-(-s,\lambda)$ are holomorphic on $\CH_{(a,+\infty)}\times U$ and $\CH_{(-\infty, b)}\times U$, respectively. Hence each $\oZ_\epsilon(s,\lambda)$ ($\epsilon=\pm$) admits a holomorphic extension to $\BC\times U$ such that \eqref{FE} remains valid. Moreover, by Proposition~\ref{prop:conv0}, the extended functions $\oZ_+(s,\lambda)$ and $\oZ_-(s,\lambda)$ on $\BC\times U$ satisfy the hypotheses of \cite[Proposition~2.8.1]{BP21}; consequently they extend to holomorphic functions on $\BC\times \BC^r$ that still satisfy \eqref{FE}. This proves the holomorphy assertion in Part~(2) of Theorem~\ref{thm:FE_m}, as well as Part~(3). The linearity assertion in Part~(2) is immediate.

It remains to prove the continuity assertion in Part~(2). Following the proof of \cite[Proposition~2.8.1]{BP21}, for any connected relatively compact open subset $U'\subset \BC^r$ containing $U$, there exist $C>0$ and a nonnegative integer $N$ such that the functions
\[
\oZ_{\epsilon,N}(s,\lambda):= e^{s^N}\oZ_\epsilon(s,\lambda)
\]
on $\CH_{(C,\infty)}\times U'$ are rapidly decreasing in vertical strips in the first variable, locally uniformly in the second variable (see \cite[Section~2.8]{BP21} for the precise definition). By Cauchy's integral formula,
\be \label{Cauchy}
\oZ_{\epsilon,N}(s,\lambda)
= \frac{1}{2\pi i}\left(
\int^\infty_{-\infty}\frac{\oZ_{\epsilon,N}(D+{\rm i}x, \lambda)}{D+{\rm i}x-s}\od\!x
- \int^\infty_{-\infty}\frac{\oZ_{-\epsilon,N}(D+{\rm i}x, \lambda)}{D+{\rm i}x+s}\od\!x
\right)
\ee
for every $D>C$ and $(s,\lambda)\in \CH_{(-D,D)}\times U'$, where $-\epsilon$ denotes $-$ if $\epsilon=+$ and $+$ if $\epsilon=-$. The continuity assertion now follows from \eqref{Cauchy}, the dominated convergence theorem, Proposition~\ref{prop:conv0}, and the continuity of the map
\[
\CU[\prec\mu]\times \pi_\tau \to \CC_\mu(N_m\bs G_m, \psi_m),\qquad (\lambda, f)\mapsto W_{f_\lambda},
\]
for every $\mu\in \CA_0^*$ (identifying $\BC^r$ with $\CA_{M,\BC}^*$ in the natural way).

Finally, Part~(4) of Theorem~\ref{thm:FE_m} is an immediate consequence of \eqref{eq:FE_m}, Proposition~\ref{prop:nonv}, and the standard properties of Artin $L$-functions.
\end{proof}

We are now in a position to prove Theorem~\ref{thm:FE_m}. By Theorem~\ref{thm2:FE_m}, the assumption of Lemma~\ref{lem:red} holds when $\overline P$ is a Borel subgroup. Hence Lemma~\ref{lem:red} implies Theorem~\ref{thm:FE_m} for principal series representations.

Now set
\[
\CN:=\Set{\lambda=(\lambda_1, \lambda_2,\dots, \lambda_r)\in\BC^r : \abs{\Re(\lambda_i)-\Re(\eta)}<\tfrac14,\ i=1,2,\ldots,r}.
\]
For $\lambda\in \CN$, the representation $\pi_\lambda\otimes |\eta|^{-\frac12}$ is nearly tempered. By the assumption on $\tau$ and the classification of tempered representations of general linear groups (see \cite{Z80} for the non-archimedean case), $\pi_\lambda$ is isomorphic to a quotient of a principal series representation $I(\xi)$. Since $\pi_\lambda$ is irreducible, its contragredient $\pi_\lambda^\vee$ is isomorphic to a quotient of $I(\tilde\xi)$. Consequently,
\[
\CW(\pi_\lambda,\psi_m)= \CW(I(\xi),\psi_m)
\quad\text{and}\quad
\CW(\pi_\lambda^\vee,\overline\psi_m) = \CW(I(\tilde\xi),\overline\psi_m).
\]
A direct computation using local $L$-parameters gives
\[
\gamma(s, \pi_\lambda, \wedge^2\otimes \eta^{-1}, \psi)
= \gamma(s, I(\xi), \wedge^2\otimes\eta^{-1}, \psi).
\]
Thus Theorem~\ref{thm2:FE_m} implies that \eqref{eq:FE_m} holds for all $(s,\lambda)\in \Omega_\CN$. Applying Lemma~\ref{lem:red} once again, we conclude that Theorem~\ref{thm:FE_m} holds. 

Finally, Theorem~\ref{thm:FE'_m} (2) follows from Theorem~\ref{thm:FE_m} (2) and Theorem~\ref{thm:MF_m}. Since Theorem~\ref{thm:FE'_m} (1) has been proved in Section \ref{sec:FE'}, this completes the proof of Theorem~\ref{thm:FE'_m}.

\section*{Acknowledgements}


The authors thank Rapha\"el Beuzart-Plessis for answering a question about \cite{BP21}. 
D. Jiang is supported in part by 
the Simons Grants: SFI-MPS-SFM-00005659 and 
SFI-MPS-TSM-00013449. 
D. Liu is supported in part by National Key R \& D Program of China No. 2022YFA1005300 and National Natural Science Foundation of China No. 12526208.  B. Sun is supported in part by National Key R \& D Program of China No. 2022YFA1005300  and New Cornerstone Science Foundation. F. Tian
is supported in part by National Key R \& D Program of China No. 2022YFA1005304.

\end{document}